\documentclass[reqno, 11pt]{amsart}
\usepackage{graphicx, color, amsmath, amsfonts, amssymb, amsthm, amscd}
\usepackage{enumerate}
\usepackage[colorlinks=true, pdfstartview=FitV, linkcolor=blue,
  citecolor=blue, urlcolor=blue,pagebackref=false]{hyperref}

\newcommand{\bbG}{\mathbb{G}}
\newcommand{\bbE}{\mathbb{E}}
\newcommand{\bbP}{\mathbb{P}}
\newcommand{\bbT}{\mathbb{T}}

\newcommand{\bbN}{\mathbb{N}}
\newcommand{\bbZ}{\mathbb{Z}}
\newcommand{\bbR}{\mathbb{R}}

\newcommand{\cI}{\ensuremath{\mathcal I}}
\newcommand{\cE}{\ensuremath{\mathcal E}}

\newcommand{\cX}{\ensuremath{\mathcal X}}
\newcommand{\cN}{\ensuremath{\mathcal N}}
\newcommand{\cA}{\ensuremath{\mathcal A}}
\newcommand{\cB}{\ensuremath{\mathcal B}}
\newcommand{\cC}{\ensuremath{\mathcal C}}
\newcommand{\cL}{\ensuremath{\mathcal L}}

\newcommand{\rmB}{{\rm B}}
\newcommand{\rmS}{{\rm S}}

\newcommand{\lra}[1]{\stackrel{{#1}}{\longrightarrow}}
\newcommand{\eqt}[1]{\stackrel{{#1}}{=}}

\newtheorem{theorem}{Theorem}[section]
\newtheorem{cor}[theorem]{Corollary}
\newtheorem{lem}[theorem]{Lemma}
\newtheorem{prop}[theorem]{Proposition}
\newtheorem{thm}[theorem]{Theorem}

\begin{document}

\title{The Williams Bjerknes Model on Regular Trees}
\author{Oren Louidor}
\address{Mathematics Department\\
 University of California Los Angeles\\
	Los Angeles, CA, USA}
\email{louidor@math.ucla.edu}
\thanks{The research of the first author was supported by a Simons Postdoctoral Fellowship}

\author{Ran J. Tessler}
\address{Mathematics Department\\
The Hebrew University of Jerusalem\\
Jerusalem, Israel}
\email{ran.tessler@mail.huji.ac.il}

\author{Alexander Vandenberg-Rodes}
\address{Mathematics Department\\
 University of California Irvine \\
	Los Angeles, CA, USA}
\email{vandenbe@math.uci.edu}

\begin{abstract}
We consider the Williams Bjerknes model, also known as the biased voter model on the $d$-regular tree $\bbT^d$, where $d \geq 3$. Starting from an initial configuration of ``healthy'' and ``infected'' vertices, infected vertices infect their neighbors at Poisson rate $\lambda \geq 1$, while healthy vertices heal their neighbors at Poisson rate $1$. All vertices act independently. It is well known that starting from a configuration with a positive but finite number of infected vertices, infected vertices will continue to exist at all time with positive probability iff $\lambda > 1$. We show that there exists a threshold $\lambda_c \in (1, \infty)$ such that if $\lambda > \lambda_c$ then in the above setting with positive probability all vertices will become eventually infected forever, while if $\lambda < \lambda_c$, all vertices will become eventually healthy with probability $1$. In particular, this yields a complete convergence theorem for the model and its dual, a certain branching coalescing random walk on $\bbT^d$ -- above $\lambda_c$. We also treat the case of initial configurations chosen according to a distribution which is invariant or ergodic with respect to the group of automorphisms of $\bbT^d$.
\end{abstract}

\maketitle

\section{Introduction and Results}
We study the {\em Williams Bjerknes model} (henceforth WB process), also known as the {\em biased voter model}, on the $d$-regular tree $\bbT = \bbT^d$ for $d \geq 3$. This a continuous time Markov process whose state space is $\cX := \{-,+\}^\bbT$, i.e. the set of all configurations (assignments) of $\pm$ to the vertices of the tree. ``$+$ vertices'' will be thought of as {\em infected}, while ``$-$ vertices'' as {\em healthy}. Starting from some initial configuration $\xi_0 \in \cX$, infected vertices infect each of their neighbors at Poisson rate $\lambda$, where $\lambda \geq 1$ is the {\em infection rate} parameter, while healthy vertices heal each of their neighbors at Poisson rate $1$. All vertices act independently. We shall denote by $\xi_t^{\xi_0, \bbT, \lambda}$ the state of this process at time $t$ and often omit some or all of the superscripts when they are clear or irrelevant. 

Formally, $(\xi_t^{\xi_0, \bbT, \lambda} :\: t \geq 0)$ is a Markov spin-system whose generator is the closure in $\cC(\cX)$ of the operator (defined on a suitable sub-space of $\cC(\cX)$).
\begin{equation}
\label{eqn:1}
	\cL f (\xi) = \sum_x  \big( 1_{\{\xi(x) = +\}} + \lambda 1_{\{\xi(x) = -\}} \big)	
		 \big| \{y \sim x :\: \xi(y) \neq \xi(x) \} \big| \big[ f(\xi^x) - f(\xi) \big] \,,
\end{equation}
where $\xi^x$ is equal to $\xi$ except at the vertex $x$ where it has the opposite sign and $x \sim y$ means that $x$ and $y$ are neighboring vertices in $\bbT$. We shall identify a configuration $\xi$ with the subset of vertices which are infected under it, i.e. the set $\{x \in \bbT :\: \xi(x) = +\}$.

This process was introduced in 1972 by Williams and Bjerknes~\cite{williams1971stochastic} as a model for tumor growth and independently by Schwartz~\cite{schwartz1977applications} in 1977 as an example of a particle system with an increasing dual. It is closely related to both the {\em voter model} (the case of $\lambda = 1$) and the {\em contact process} (healing rates are fixed and do not depend on the number of healthy neighbors). As such it exhibits behavior which is similar to both models (this will be further discussed below). For standard texts on all these models see \cite{liggett2004interacting, liggett1999stochastic}

The main question in this model, both from a mathematical and a biological point of view, is that of survival. Namely, starting from a finite non-empty initial configuration $\xi_0$, i.e., $0 < |\xi_0| < \infty$ (where $|\xi_0|$ is the cardinality of $\xi_0$), whether infected sites will continue to exist at all times or become extinct. As was noticed by Williams and Bjerknes, observed at the times of transition: $0=\tau_0, \tau_1, \tau_2, \dots$, the process $(|\xi_{\tau_k}| :\: k=0,1,\dots)$ is just a nearest-neighbor random walk on $\bbZ_+$ with an absorbing state at $0$ and drift 
\begin{equation}
\label{eqn:303}
	-1 \frac{1}{\lambda+1} + 1 \frac{\lambda}{\lambda+1} = \frac{\lambda-1}{\lambda+1} \,.
\end{equation}
Therefore {\em global} survival, i.e.
\begin{equation}
\label{eqn:301}
	\Omega_g^{\xi_0} := \big\{ \sup \{t \geq 0 :\: 
		\xi^{\xi _0}_t \neq \emptyset\} = \infty \big\} \,,
\end{equation}
has probability 
\begin{equation}
\label{eqn:2}
	\bbP(\Omega_g^{\xi_0}) = 1-\lambda^{-|\xi_0|} \,,
\end{equation}
which for finite non empty $\xi_0$, is positive if and only if $\lambda > 1$ (the reason for the term ``global'' will become apparent shortly). In other words, the threshold for the possibility of global survival is $\lambda_g = 1$ regardless of the underlying graph, as long as it is infinite, connected and has a bounded degree (this can be relaxed, but some restrictions are needed to ensure that the process is well-defined).

In the lattice case, based on numerical simulations Williams and Bjerknes predicted that once the infection survives, the set of infected sites will ''roughly'' look like an ever growing ``blob'' around the initially infected vertex. This was proved by Bramson and Griffeath in 1980 \cite{BG1, BG2} who gave a shape theorem with a linear rate for the subset of infected sites -- for $\bbZ^d$ in all $d \geq 1$ and any $\lambda > 1$. (This is similar to the shape theorem for the Richardson Growth Model, which was proved by Richardson \cite{richardson1973random} and Kesten \cite{kesten1973contribution}.) 
Thus, in particular on $\bbZ^d$ for all $\lambda > 1$, global survival implies {\em complete} survival, namely
\begin{equation}
\label{eqn:3}
	\Omega_c^{\xi_0} := \big\{ \sup \{t \geq 0 :\: \xi_t^{\xi_0} \not \ni x\} < \infty \big\} \,.
\end{equation}
where $x$ is any vertex of $\bbZ^d$. Notice that except for an event of zero probability, on $\Omega_c^{\xi_0}$ eventually all vertices will become infected, hence the choice of $x$ is immaterial in the above definition.

On the $d$-regular tree, the situation is more intricate and so far was less understood. Madras, Schinazi and Durrett \cite{MS1} showed that for $d \geq 3$, survival can be global but not complete. More precisely, for $\xi_0 \in \cX$ and any $x \in \bbT$, define {\em local} survival as the event
\begin{equation}
\label{eqn:302}
	\Omega_l^{\xi_0} := \{ \sup \{t \geq 0 :\: \xi_t^{\xi_0} \ni x\} = \infty \} \,,
\end{equation}
Then for all $d \geq 3$, there exists $\lambda'$ strictly higher than $\lambda_g = 1$, such that for any finite non-empty $\xi_0$, if $\lambda \in (1, \lambda')$ then $\bbP(\Omega^{\xi_0}_l) = 0$. In other words, if $\lambda \in (1, \lambda')$ the infection can survive, but it must eventually ``drift to infinity''. Letting 
\begin{equation}
	\lambda_l(\bbT) := \inf \big\{\lambda > 0 :\: \bbP(\Omega^{\xi_0}_l) > 0 \big\}
\end{equation}
denote the threshold value for local survival, where $\xi_0$ is any finite non-empty configuration, whose precise value is immaterial, they are able to show that
\begin{equation}
	\lambda_l(\bbT^d) \geq \frac{d}{2\sqrt{d-1}} \,.
\end{equation}
However, it was not clear whether $\lambda_l(\bbT) < \infty$ nor what exactly happens above this threshold. More precisely, if we define
\begin{equation}
	\lambda_c(\bbT) := \inf \big\{\lambda > 0 :\: \bbP(\Omega^{\xi_0}_c) > 0 \big\}
\end{equation}
then it is not clear whether $\lambda_c(\bbT) < \infty$ and whether its value coincides with that of $\lambda_l(\bbT)$ or strictly larger than it. Had $\lambda_l(\bbT) < \lambda_c(\bbT) < \infty$ been the case, there would have been three phases for the model: global but not local nor complete survival, local but not complete survival and then complete survival.

The notion of local survival and the existence of an intermediate phase where survival is global but not local was first observed by Pemantle in the context of the contact process on trees \cite{pemantle1992contact}. By finding upper and lower bounds on the infection thresholds for global, resp. local survival he was able to conclude that there is  an intermediate regime for $\bbT^d$ when $d \geq 4$. Liggett~\cite{liggett1996multiple} and then~\cite{stacey1996existence} showed that this is also true for $d=3$. 

By adapting the martingale methods of Pemantle, one can fairly easily obtain bounds on the threshold values $\lambda_l(\bbT)$ and $\lambda_c(\bbT)$. 
\begin{prop}
\label{prop:Bounds}
Let $d \geq 3$.  
\begin{enumerate}
  \item \label{item:CBOL1} 
    $\frac{d}{2\sqrt{d-1}} \leq \lambda_l(\bbT^d) \leq \min \Big (2d ,\, \frac{4d}{(\sqrt{d-1}-4) \vee 0} \Big)$.
  \item \label{item:CBOL2} 
  	$\lambda_c(\bbT^d) \leq (d-1) \, \vee \, \lambda_l(\bbT^d)$. 
\end{enumerate}
\end{prop}
This shows that both local and complete survival occur for large enough values of $\lambda$, but does not settle the question of whether there is a second intermediate phase of local but not complete survival. The lower bound in part~\ref{item:CBOL1} is the same as the one obtained in~\cite{MS1}. However, the martingale approach used here seems more robust, as it does not rely on the tree isotropy, which is exploited in~\cite{MS1}. Therefore, it could be used to handle other tree-like graphs which are less regular (e.g. a realization of a super-critical Galton-Watson process). It should be noted that the argument leading to part~\ref{item:CBOL2} of the proposition, can be applied to $\bbZ^d$ as well. In this case one gets $\lambda_c(\bbZ^d) = \lambda_g = 1$ for all $d \geq 1$, thereby providing a very short proof for~\eqref{eqn:3}, albeit without a shape theorem.

\medskip
It requires much more work to show:
\begin{thm}
\label{thm:complete} 
For all $d \geq 3$ we have $\lambda_c(\bbT^d) = \lambda_l(\bbT^d)$. 
\end{thm}
The proof of this theorem constitutes the main part of this paper. The theorem implies that the only possibility for local but not complete survival is when $\lambda = \lambda_l(\bbT) = \lambda_c(\bbT)$. We conjecture that this is not the case and that in fact at this $\lambda$ survival can only be global. This is the case in the contact process~\cite{zhang1996complete, lalley1998limit}. As an immediate corollary we get the following characterization of all possible weak limits of $\xi_\cdot$. In what follows, we naturally endow the space $\cX$ with the product topology and product $\sigma$-algebra.
\begin{cor}
\label{cor:CompelteConvergenceXi}
Let $d \geq 3$ and $\lambda > \lambda_l(\bbT^d)$. For all $\xi_0 \in \cX$ as $t \to \infty$,
\begin{equation}
\label{e:2}
  \bbP(\xi_t^{\xi_0} \in \cdot) \Rightarrow 
	\bbP(\Omega_g^{\xi_0}) \, \delta_{\bbT} \ + \ 
		\big(1 - \bbP(\Omega_g^{\xi_0}) \big) \delta_\emptyset \, .
\end{equation}
In particular $\delta_0$ and $\delta_{\bbT}$ are the only extremal invariant measures for the model above $\lambda_l(\bbT)$.
\end{cor}
\eqref{e:2} is an analog of the Complete Convergence Theorem for the contact process on $\bbT$, conjectured by Pemantle and first proved by Zhang~\cite{zhang1996complete} and reproved in a simpler way by Schonmann and Salzano~\cite{salzano1998new}. Here $\delta_\bbT$ plays the role of the upper invariant measure of the contact process (i.e. the limiting measure of the process when started from the all $+$ configuration). When $\lambda \in (1, \lambda_l(\bbT)]$ it is not clear whether aside from $\delta_\emptyset$ and $\delta_\bbT$ there are other extremal invariant measures. We conjecture that this is the case, as it is for the contact process~\cite{durrett1995intermediate} below the threshold for local survival and the case when $\lambda=1$~\cite{liggett1999stochastic}.

Another consequence of Theorem~\ref{thm:complete} is the process ``mostly'' fixates. More formally, for $\xi_0 \in \cX$ and $x \in \bbT$ define the {\em fixation} event as
\begin{equation}
	\Omega_f^{\xi_0} := \big\{ \sup \{t \geq 0 :\: 
				\xi^{\xi_0}_{t-}(x) \neq \xi^{\xi_0}_t(x) \} < \infty \big\} \,.
\end{equation}
Then,
\begin{cor}
\label{cor:Fixation}
Fix $d \geq 3$. 
\begin{enumerate}
	\item \label{item:CS1}
 		If $\lambda > \lambda_l(\bbT^d)$ then $\bbP \big(\Omega_f^{\xi_0} \big) = 1$,
 		for any $\xi_0 \in \cX$.
 	\item \label{item:CS2}
	 	If $1 \leq \lambda < \lambda_l (\bbT^d)$ then 
	  	$\bbP \big(\Omega_f^{\xi_0} \big) = 1$, for any finite $\xi_0 \in \cX$. 
\end{enumerate}
\end{cor}

Duality plays an important role in the analysis of particle systems (see, for example \cite{harris1976class, griffeath1979additive}). At the same time, the dual processes are often of interest by themselves. When $\lambda \geq 1$, a dual for the WB process, which was exploited time and again in the past, is a certain (continuous time) branching coalescing random walk (henceforth the BCRW process), which we now describe. Like $\xi_\cdot$, this process takes value in the space $\cX$ of all $\pm$ configurations on $\bbT$. However, this time we interpret a ``$+$ vertex'' as occupied by a particle, while a ``$-$ vertex'' as vacant. Starting from an initial configuration $\hat{\xi}_0$,  particles independently move to each of their neighbors at rate $1$ and give birth (branch) to a new particle at each of their neighbors at rate $\lambda-1$. If a vertex to which a particle moved or branched was already occupied by a particle, the two particles coalesce. We shall denote this process by $(\hat{\xi}_t^{\hat{\xi}_0, \bbT, \lambda} :\: t \geq 0)$. Formally, its generator is the closure of 
\begin{equation}
\label{eqn:10}
\begin{split}
	\hat{\cL} f (\hat{\xi}) = \sum_{x \sim y} & \Big(
	 	1_{\{\hat{\xi}(x) = +, \hat{\xi}(y) = -\}} \big( \big[ f(\hat{\xi}^{xy}) - f(\hat{\xi}) \big] + (\lambda - 1)
			 	\big[ f(\hat{\xi}^{y}) - f(\hat{\xi}) \big] \big)  \\
	& \quad + 1_{\{\hat{\xi}(x) = +, \hat{\xi}(y) = +\}} \big( \big[ f(\hat{\xi}^{x}) - f(\hat{\xi}) \big] + 
			 	\big[ f(\hat{\xi}^{y}) - f(\hat{\xi}) \big] \big) \Big) \,.
\end{split}
\end{equation}
As before $\hat{\xi}^x$ is $\hat{\xi}$ with the sign at $x$ flipped, while $\hat{\xi}^{xy}$ is $\hat{\xi}$ with the sign flipped both at $x$ and at $y$. 

There are two known duality relations between $\xi_\cdot$ and $\hat{\xi}_t$. The first one which is more standard, can be read immediately from the graphical representation of the model. The second was discovered by Sudbury and Lloyd~\cite{sudbury1997quantum} and involves {\em $p$-thinning} of configurations, whereby each $+$ vertex becomes a $-$ vertex with probability $1-p$ and kept $+$ with probability $p$, independently of other vertices. We shall make use of both of these relations in the proofs, but they are not needed in order to state the results concerning $\hat{\xi}_\cdot$ and therefore we shall defer their precise formulation to subsection~\ref{sub:duality}. 

For $p \in [0,1]$, let $\nu_p$ denote the Bernoulli($p$)-product measure on $\cX$. Using any of the two duality relations, the previous results on $\xi_\cdot$ immediately give,
\begin{thm}
\label{thm:CC_BCRW}
Let $d \geq 3$. If $\lambda \in [1, \lambda_l(\bbT^d))$ then for any finite $\hat{\xi}_0 \in \cX$
as $t \to \infty$,
\begin{equation}
  \label{e:4}
  \bbP(\hat{\xi}_t^{\hat{\xi}_0} \in \cdot) \Rightarrow \delta_\emptyset \,.
\end{equation}
If  $\lambda > \lambda_l(\bbT^d)$ then for any $\hat{\xi}_0 \in \cX$ as $t \to \infty$,
\begin{equation}
  \label{e:3}
  \bbP(\hat{\xi}_t^{\hat{\xi}_0} \in \cdot) \Rightarrow 
  		1_{\{\hat{\xi}_0 \neq \emptyset\}} \; \nu_{1-1/\lambda} 
  		+ 1_{\{\hat{\xi}_0 = \emptyset\}} \; \delta_\emptyset \, .
\end{equation}
In particular the only extremal invariant measures for BCRW above $\lambda_l(\bbT)$ are $\delta_\emptyset$ and $\nu_{1-1/\lambda}$. 
\end{thm}
The cases of $\lambda = \lambda_l(\bbT^d)$ and any initial configuration $\hat{\xi}_0$, and $\lambda \in (1,\lambda_l(\bbT^d))$ and infinite $\hat{\xi}_0$ -- remain open, as they do for $\xi_\cdot$.

As mentioned, when $\lambda \in (1, \lambda_l(\bbT^d))$ and the initial configuration is chosen  according to a distribution which puts mass on infinite configurations, then it is an open problem to characterize the set of possible weak limits for both $\xi_\cdot$ and $\hat{\xi}_\cdot$. Nevertheless, if the initial configuration is invariant or even ergodic, with respect to the group of automorphisms of $\bbT^d$, then such a characterization is possible. 

More precisely, denote by $\cI$ the set of probability measures on $\cX$ which are invariant under all automorphisms of $\bbT^d$. The subset of $\cI$ of all measures which are in addition ergodic will be denoted by $\cE$. For a configuration $\xi$, a ({\em connected}) {\em component} is a maximal subset of vertices $U$ of $\bbT$, for which the induced sub-graph is connected and such that all vertices in $U$ have the same sign under $\xi$. We shall call a component infected, if its vertices are infected under $\xi$. Then we have the following:
\begin{thm}
\label{thm:Automorphism}
Let $d \geq 3$ and $\lambda > 1$.
\begin{enumerate}
\item \label{part:AU1}
	If $\bbP(\xi_0 \in \cdot) \in \cI$ then 
	\begin{equation}
		  \bbP \big( \Omega^{\xi_0}_f \big) = 1 \,.
	\end{equation}
	In particular, any automorphism-invariant stationary distribution for $\xi_\cdot$ is a convex combination of $\delta_\emptyset$ and $\delta_{\bbT}$.
\item \label{part:AU2}
	If $\bbP(\xi_0 \in \cdot) \in \cE \setminus \{\delta_\emptyset\}$ then 
	\begin{equation}
		\label{eqn:516}
		  \bbP \big( \Omega^{\xi_0}_c \big) = 1 \,.
	\end{equation}
	In this case, infinite infected components are formed in finite time $\bbP$-almost surely.
	In particular, the only automorphism-ergodic stationary distributions for $\xi_\cdot$ are  $\delta_\emptyset$ and $\delta_{\bbT}$.
\end{enumerate}
\end{thm}
A similar theorem can be derived for $\hat{\xi}$. It should be noted that the proof of Theorem~\ref{thm:Automorphism} applies to a much larger class of vertex transitive graphs.

\subsection{Outline of the paper}
The remainder of the paper is organized as follows. In section~\ref{sec:Perliminaries} we recall some known facts about the WB and BCRW processes as well as introduce most of the notation which will be used later in the proofs. In section~\ref{sec:MainProofs} we prove Theorems~\ref{thm:complete} and~\ref{thm:CC_BCRW} as well as Corollaries~\ref{cor:CompelteConvergenceXi} and~\ref{cor:Fixation}. Section~\ref{sec:ProofOfBounds} includes the proof of Proposition~\ref{prop:Bounds} and finally section~\ref{sec:ProofOfAuto} is devoted to the proof of Theorem~\ref{thm:Automorphism}.

\section{Preliminaries and Notation}
\label{sec:Perliminaries}
In this section we setup some additional notation which will often be used in the sequel as well as collect some well known facts about the process and its dual. Any future use of these facts will be accompanied by a proper reference to this section. Consequently, the reader who is familiar with the model can skim through this section quickly or skip it altogether, without much risk of getting lost later on.

\subsection{Graphs}
We will mostly be concerned with the $d$-regular tree $\bbT$, although occasionally we shall use other graphs $\bbG = (V,E)$. We shall identify sub-graphs with their corresponding edge-set and vertex-set. For example, for a set of vertices $U$ we may write $U \subseteq \bbT$ to indicate that $U$ is a subset of the vertex set of $\bbT$ and in this case treat $U$ also as the sub-graph of $\bbT$ induced by the vertices in $U$. We shall distinguish one vertex of $\bbT$ to be called the origin and denoted $0$. 

Although in the definition of $\xi_\cdot$ the underlying graph need not be directed,  it will be convenient to think of the edges of $\bbT$ as oriented such that each vertex will have exactly one predecessor, its {\em parent} and $d-1$ successors -- its {\em children} (formally, we fix an end of $\bbT$ and define the parent of $x$ as the first vertex after $x$ on the ray from $x$ which belongs to this end). For a vertex $x \in \bbT$, we let $\bbT_x$ denote the subtree, in the above orientation, rooted at $x$. 

The graph distance will be denoted by $\rho$. For $x \in \bbG$ and $r > 0$ we denote by $\rmB_x(r)$ the closed ball of radius $r$ around $x$ in this metric, namely $\rmB_x(r) := \{y \in \bbG :\: \rho(x,y) \leq r\}$ and set $\rmS_x(r) := \rmB_x(r) \setminus \rmB_x(r-1)$. Given a subset of vertices $U \subseteq \bbG$, we denote by $\partial_{\bbG} U$ the set of edges in $\bbG$ with exactly one endpoint in $U$.

\subsection{WB and BCRW on general underlying graphs and boundary conditions}
\label{sub:OnGeneralGraphs}
The definition of $\xi_\cdot$ in~\eqref{eqn:1} and $\hat{\xi}_\cdot$ in~\eqref{eqn:10} can, of course, be extended to any underlying graph $\bbG = (V,E)$ with a bounded degree (as mentioned, this can be relaxed). In this case the state space is $\cX^{\bbG} := \{+, -\}^{\bbG}$ and the Williams-Bjerknes process for such graph, initial configuration $\xi_0 \in \cX^{\bbG}$ and infection parameter $\lambda \geq 1$ will be denoted by $\big(\xi_t^{\xi_0, \bbG, \lambda} :\: t\geq 0 \big)$. Similarly, the corresponding branching coalescing random walk will be denoted by $\big(\hat{\xi}_t^{\xi_0, \bbG, \lambda} :\: t\geq 0 \big)$. As mentioned in the introduction, we shall often omit some or all of the superscripts.

The {\em inclusion time} of a subset of vertices $U \subseteq \bbG$ will be used often. For $\xi_\cdot$ it is defined as 
\begin{equation}
	\tau_U := \inf \{ t \geq 0 :\: \xi_t \supseteq U \} \,.
\end{equation}
Similarly $\hat{\tau}_U$ will denote the inclusion time of $U$ by $\hat{\xi}_\cdot$.

We shall often treat several instances of $\xi_\cdot$ and $\hat{\xi}_\cdot$ corresponding to different $(\xi_0, \bbG, \lambda)$ at the same time. In this case, it will be useful to decorate all events and random variables pertaining to a certain instance with the same superscripts and accents used to denote the process itself. For example, we may write $\Omega_g^{\xi_0, \bbG, \lambda}$ for the event of global survival for $\xi_\cdot^{\xi_0, \bbG, \lambda}$ or $\hat{\tau}_x^{\xi_0, \bbG}$ for the inclusion time of $x$ by $\hat{\xi}^{\xi_0, \bbG}_\cdot$.

If $\bbG$ is a sub-graph of a larger graph $\bbG'$, we may often want to put {\em boundary conditions} on the vertices of $\bbG' \setminus \bbG$. Given $\zeta \in \cX^{\bbG' \setminus \bbG} = \{-,+\}^{\bbG' \setminus \bbG}$, the process $\xi_\cdot$ with boundary conditions $\zeta$ evolves as before, only that the sign of vertices in $\bbG' \setminus \bbG$ remain fixed according to $\zeta$. Thus vertices in $\bbG' \setminus \bbG$ cannot be infected nor healed, but they continue to infect or heal their neighboring vertices in $\bbG$ at the usual rates. 

It will be convenient to suppose that sub-graphs can possibly ``come'' with boundary conditions and we shall write $\bbG^\zeta$ to mean that $\bbG$ ''comes'' with boundary conditions $\zeta$ on $\bbG' \setminus \bbG$. Writing just $\bbG$ means that there are no boundary conditions associated with $\bbG$. In practice, we shall only use either the {\em $+$ boundary conditions}, by which we mean that $\zeta = \delta_{\bbG' \setminus \bbG}$ or the {\em $-$ boundary conditions}, by which we mean that $\zeta = \delta_\emptyset$. In these cases we shall write either $\bbG^+$ or $\bbG^-$. Furthermore, if $\bbG'$ is not specified it will be assumed to be $\bbT$.
For example, $\xi_\cdot^{\xi_0, \bbT^+_0, \lambda}$ is the WB process on $\bbT_0$ with $+$ boundary conditions on $\bbT \setminus \bbT_0$.

Boundary conditions will also be used for $\hat{\xi}$, although here we need to clarify what they mean exactly. Given $\zeta \in \cX^{\bbG' \setminus \bbG}$ as before, $\hat{\xi}_\cdot$ on $\bbG$ with boundary conditions $\zeta$ evolves as $\hat{\xi}_\cdot$ does, only that particles which reach a $-$ vertex in $\bbG' \setminus \bbG$ disappear, while particles which reach a $+$ vertex in $\bbG' \setminus \bbG$ stay there forever. No particles are initially placed in any of the vertices of $\bbG' \setminus \bbG$. We shall see in subsection~\ref{sub:duality} why this definition is useful.

\subsection{The graphical representation}
\label{sub:Graphical}
The use of a graphical representation for describing the evolution of particle systems, originally due to Harris~\cite{harris1978additive}, is now a standard tool in their analysis. A more detailed account of this construction can be found in~\cite{griffeath1979additive}. Let a graph $\bbG = (V,E)$ and an infection parameter $\lambda \geq 1$ be given. Consider the set $D_{\bbG} = V \times \bbR_+$ which we think of as embedded in the plane as a disjoint collection of vertical rays, one for each vertex in $V$, starting at some point on the $x$-axis and going upwards. An element $(v,t)$ of $D_{\bbG}$ where $v \in V$ and $t \geq 0$ is therefore identified with the point on the ray corresponding to $v$ at height $t$ above the $x$-axis. We think of the second coordinate $t$ as time.

With each ordered pair of neighboring vertices $u \sim v$ in $\bbG$, we associate two Poisson point processes on $\bbR_+$: $\cN_{u,v}^\bullet$, $\cN_{u,v}^\circ$. The former has intensity measure $(\lambda-1) dt$ and the latter $1 dt$. Now fix a realization of all these processes. For each point $t$ in $\cN_{u,v}^\bullet$ we add to $D_{\bbG}$ a horizontal segment between $(u,t)$ and $(v,t)$, which we think of as oriented from $(u,t)$ to $(v,t)$. Similarly, for each point $s$ in $\cN_{u,v}^\circ$ we add a horizontal segment between $(u,s)$ and $(v,s)$, which we think of as oriented from $(u,s)$ to $(v,s)$, but just below $(v,s)$ we make a hole in the ray corresponding to $v$. The set $D_{\bbG}$ along with all oriented segments and holes will be denoted $D_{\bbG, \lambda}$. This is, of course, a random subset of $\bbR \times \bbR_+$.

Given a realization of $D_{\bbG, \lambda}$, a path from $(u,s)$ to $(v,t)$, where 
$0 \leq s \leq t$ and $u,v \in U$, is a self-avoiding curve from $(u,s)$ to $(v,t)$ which is also a subset of $D_{\bbG, \lambda}$ and adheres to the orientation of all rays and segments. In other words, it can only go upwards on a ray and in the direction of the segment on a segment and cannot pass through holes. If $\gamma$ is such a path we shall write $\gamma: (u,s) \lra{D_{\bbG, \lambda}} (v,t)$. 

The following easy to see relation explains the connection between $\xi_\cdot$ and this graphical representation. Recall that for a subset of vertices $A \subseteq \bbG$ the WB process on $\bbG$ with infection parameter $\lambda$ and initial configuration $\xi_0 = A$ is denoted by $\xi_\cdot^{A, \bbG, \lambda}$. Then, for any $A, B \subseteq \bbG$,
\begin{equation}
\label{eqn:15}
	\bbP(\xi_t^{A, \bbG, \lambda} \cap B \neq \emptyset) = 
		\bbP \big( \exists \gamma:(u,0) \lra{D_{\bbG, \lambda}} (v,t) \ \text{ such that } u \in A, \, v \in B \big) \,.
\end{equation}
In other words, if we set 
\begin{equation}
\label{eqn:16}
	\xi_t^{A, \bbG, \lambda} := \{v \in \bbG :\: \exists \gamma: (u, 0)  \lra{D_{\bbG, \lambda}} (v,t) \ \text{ such that }
		u \in A\} \,,
\end{equation}
then $(\xi_t^{A, \bbG, \lambda} :\: t \geq 0)$ is the Williams-Bjerknes process for $A, \bbG, \lambda$.

If $\bbG\subset\bbG'$ has boundary conditions $\zeta\in \cX^{\bbG'\setminus\bbG}$, then with the ordered neighbors $u\sim v$ with $u\in \bbG'\setminus\bbG$, $v\in\bbG$ we also associate the point processes $\cN^\bullet_{u,v},\cN^\circ_{u,v}$. Using the same construction as above we define the set $D_{\bbG',\lambda}$, and now
\begin{equation}
\label{eqn:161}
	\xi_t^{A, \bbG^\zeta, \lambda} := \{v \in \bbG :\: \exists \gamma: (u, 0)  \lra{D_{\bbG', \lambda}} (v,t) \ \text{ such that }
		u \in A\cup\zeta\} \,,
\end{equation}
is the Williams-Bjerknes process with boundary conditions $\zeta$. 
The usefulness of this graphical representation will become apparent in the next subsections.
\subsection{Coupling}
The graphical representation gives rise to a natural coupling between instances of $\xi^{A, \bbG, \lambda}_\cdot$ for different initial configurations $A$, underlying sub-graphs $\bbG \subseteq \bbG'$ and infection parameters $\lambda \geq 1$. This is because there is a natural way to couple $D_{\bbG, \lambda}$ for different $\bbG$'s and $\lambda$'s and in light of~\eqref{eqn:16}. From this, for example, one can immediately get the following {\em monotonicity} (or {\em attractiveness}) property. If $1 \leq \lambda \leq \lambda'$ and $A \subseteq A' \subseteq \bbG$ then under the above coupling
\begin{equation}\label{monotonicity}
	\xi_t^{A, \bbG, \lambda} \ \leq \ 
		\xi_t^{A', \bbG, \lambda'} \quad \mbox{for all }t\geq 0,
\end{equation}
where the comparison is by the standard partial ordering on $\{-,+\}^{\bbG}$. This also extends to the case of graphs with boundary conditions in an obvious way. The monotonicity property will be used so frequently in the proofs to follow, that we shall often not explicitly state it.

\subsection{Duality}
\label{sub:duality}
If instead of~\eqref{eqn:16} we set
\begin{equation}
\label{eqn:18}
	\hat{\xi}_t^{B, \bbG, \lambda} := \{u \in \bbG :\: \exists \gamma: (u, 0) \lra{D_{\bbG, \lambda}} (v,t) \ \text{ such that } v \in B\} \,,
\end{equation}
then~\eqref{eqn:15} can be rewritten as 
\begin{equation}
\label{eqn:19}
	\bbP(\hat{\xi}_t^{B, \bbG, \lambda} \cap A \neq \emptyset) = 
		\bbP \big( \exists \gamma:(u,0) \lra{D_{\bbG, \lambda}} (v,t) \ \text{ such that } u \in A, \, v \in B \big) \,.
\end{equation}
and therefore
\begin{equation}
\label{eqn:201}
	\bbP(\xi_t^{A, \bbG, \lambda} \cap B \neq \emptyset) \, = \,
		\bbP(\hat{\xi}_t^{B, \bbG, \lambda} \cap A \neq \emptyset) 
		\quad \text{ for any } A,B \subseteq \bbG \,.
\end{equation}
Since the distribution of $\cN_{u,v}^{\bullet}$ and $\cN_{u,v}^{\circ}$ is invariant under time reversal, reading $D_{\bbG, \lambda}$ from time $t$ down to time $0$ (formally applying the transformation $(u,s) \mapsto (u, t-s)$ to $D_{\bbG, \lambda}$), we see that $(\hat{\xi}_t^{B, \bbG, \lambda} :\: t \geq 0)$ is distributed as the (continuous-time) branching coalescing random walk whose generator was described in~\eqref{eqn:10}, with underlying graph $\bbG$, initial configuration $B$ and parameter $\lambda$. Thus~\eqref{eqn:201} gives one duality relation between $\xi_\cdot$ and $\hat{\xi}_\cdot$.

In the presence of boundary conditions $\zeta$ on $\bbG' \setminus \bbG$ where $\bbG' \supseteq \bbG$, we can set for $B \subseteq \bbG$
\begin{equation}
\label{eqn:181}
	\hat{\xi}_t^{B, \bbG^\zeta, \lambda} := \{u \in \bbG\cup\zeta :\: \exists \gamma: (u, 0) \lra{D_{\bbG', \lambda}} (v,t) \ \text{ such that } v \in B\}  \,.
\end{equation}
which yields a process whose distribution is that of the BCRW in the presence of boundary conditions, as described in the end of subsection~\ref{sub:OnGeneralGraphs}. 
In this case relation \eqref{eqn:201} becomes
\begin{equation}
\label{eqn:202}
	\bbP(\xi_t^{A, \bbG^\zeta, \lambda} \cap B \neq \emptyset) \, = \,
		\bbP(\hat{\xi}_t^{B, \bbG^\zeta, \lambda} \cap (A \cup \zeta) \neq \emptyset) 
		\quad \text{ for any } A,B \subseteq \bbG \,.
\end{equation}
In particular for $-$ boundary conditions~\eqref{eqn:201} is still valid (with $\bbG^-$ replacing $\bbG$), while for $+$ boundary conditions, we can rewrite~\eqref{eqn:202} as 
\begin{equation}
\label{eqn:203}
	\bbP(\xi_t^{A, \bbG^+, \lambda} \cap B \neq \emptyset) \, = \,
		\bbP \big( \hat{\xi}_t^{B, \bbG^+, \lambda} \cap A  \neq \emptyset
			\ \text{ or } \ \exists s \leq t :\: \hat{\xi}^{B, \bbG^+, \lambda}_s \nsubseteq \bbG \big)
		\quad \text{ for } A,B \subseteq \bbG \,.
\end{equation}

\medskip 
To describe the second duality relation between $\xi_\cdot$ and its dual, we have to define the notion of {\itshape thinning}. Fix $p \in [0,1]$. For a configuration $\xi \in \cX$ we define the $p$-thinning $\xi^{(p)}$ of $\xi$ as the random configuration obtained from $\xi$ by independently flipping the sign of every $+$ vertex with probability $1-p$ and retaining it with probability $p$. 

The following remarkable relation is due to Sudbury and Lloyd ~\cite[Theorem 13]{sudbury1997quantum}. For any $\lambda \geq 1$,
\begin{equation}
\label{eqn:221}
	\hat{\xi}^{( \xi_0^{(p)})}_t \ \eqt{d} \ 
		\big( \xi^{\xi_0}_t \big)^{(p)}  \quad \text{where } p = 1 - \lambda^{-1} \,.
\end{equation}
Note that $p = \bbP(\Omega_g^0)$ by ~\eqref{eqn:2}.

\subsection{Additional notation}
As usual, $C, C'$ will denote positive constants whose value may change from one use to another.

\section{Proof of Theorems~\ref{thm:complete},~\ref{thm:CC_BCRW} and Corollaries}
\label{sec:MainProofs}
In this section we prove Theorem~\ref{thm:complete},~\ref{thm:CC_BCRW} and Corollaries~\ref{cor:CompelteConvergenceXi} and~\ref{cor:Fixation}. The proof of Theorem~\ref{thm:complete} is essentially linear. It consists of a sequence of lemmas, one derived from the other with the theorem following from the last. Nevertheless, to put some hierarchical structure in the proof, we have split it into two main steps which are stated in the next subsection as key lemmas. They are of interest on their own. The proofs of these lemmas are deferred to subsections~\ref{sub:ProofOfZhang} and~\ref{sub:Hitting}, so that we can first show how the theorem follows from them -- this is done in the subsection~\ref{sub:ProofOfThm}. In this subsection we also prove the two corollaries and Theorem~\ref{thm:CC_BCRW}. They are only a short step once the theorem is established.

\subsection{Key Lemmas}
The first key step is an analog of Zhang's Lemma for the contact process on regular trees~\cite[Proposition 5]{zhang1996complete}. It is the main step in Zhang's proof for the Complete Convergence Theorem in this setting. The proof was later simplified by Schonmann and Salzano~\cite[Proposition 1]{salzano1998new} and our arguments are essentially an adaption of the latter to this model.
\begin{lem}[An analog of Zhang's Lemma]
\label{lem:Zhang}
Fix $d \geq 3$. For $\lambda > \lambda_l(\bbT)$,
\begin{equation}
\label{eqn:13}
  \inf_{t \geq 0} \bbP(0 \in \hat{\xi}^{0, \bbT_{0}^-}_t) > 0 \,.
\end{equation}
\end{lem}

Next, we need tail estimates on the distribution of the inclusion time of a neighboring vertex for $\hat{\xi}_\cdot$. The lemma shows that the tail of this distribution decays faster than any polynomial. We believe that this is not optimal, but for the sake of showing complete survival this is enough. 
\begin{lem}[Super-polynomial decay for inclusion times]
\label{lem:hitting}
Fix $d \geq 3$ and $\lambda > \lambda_l(\bbT)$. For any $x \sim y$ neighboring vertices of $\bbT$,
\begin{equation}\
\label{eqn:14}
	\lim_{t \to \infty} \frac{\log \textsl{}\bbP(\hat{\tau}_y^{x, \bbT} > t)}{\log t} = -\infty \,.
\end{equation}
\end{lem}

\medskip
\subsection{Proof of the theorem and corollaries}
\label{sub:ProofOfThm}
\begin{proof}[Proof of Theorem~\ref{thm:complete}]
Fix $d \geq 3$ and $\lambda > \lambda_l(\bbT)$. It is clearly enough to prove 
\begin{equation}
\label{eqn:100}
	\bbP \big(\xi_t^{0, \bbT} \ni 0 \text{ for all } t \geq 0 \big) > 0 \,.
\end{equation}
Indeed, since $\{\xi_t^{0, \bbT} \ni 0 \text{ for all } t \geq 0 \big \} \cap \big(\Omega_c^{0, \bbT}\big)^{\rm c}$ has zero probability, this shows that $\bbP(\Omega_c^{0, \bbT}) > 0$ as required.

\eqref{eqn:100} follows from the seemingly weaker statement,
\begin{equation}
\label{eqn:101}
	\bbP \big(\xi_t^{0, \bbT_0^+} \ni 0 \text{ for all } t \geq 0 \text{ large enough} \big) > 0 \,.
\end{equation}
To see this, notice that \eqref{eqn:101} implies that there exists $\epsilon > 0$, $s > 0$ such that
\begin{equation}
	\bbP \big(\xi_t^{0, \bbT_0^+} \ni 0 \text{ for all } t \geq s \big ) > \epsilon \,.
\end{equation}
Therefore, writing $\tilde{\bbT}$ for $\bbT$ rooted at $0$, and recalling section \ref{sub:Graphical}, we have
\begin{equation}
\begin{split}
	\bbP \big(\xi_t^{0, \bbT} \ni 0 \text{ for all } t \geq 0 \big) 
		& \geq \bbP \Big(
			\cup_{x \sim 0} 
			\big\{ \xi_t^{x, \tilde{\bbT}_x^+} \ni x \text{ for all } t \geq s
				\,,\,\, \cN_{x,0}^\circ([0,s]) = 0 \big\} \, \Big)	 \\
		& \geq (\epsilon e^{-s})^d > 0 \,.
\end{split}
\end{equation}

To establish~\eqref{eqn:101} we will show that for all $s \geq 0$
\begin{equation}
\label{eqn:102}
	\bbP \big( \xi_t^{0, \bbT^+_0} \not \ni 0 \mbox{ for some } t \in [s^2,(s+1)^2) \big)
		\leq C s^{-2} \,.
\end{equation}
Since these probabilities are summable in $s=1,2, \dots$, the Borel-Cantelli Lemma will imply that~\eqref{eqn:101} holds (with probability $1$).

To this end, fix $s \in \bbN$ and let $t_0 < t_1 < \cdots < t_{s^4}$ be a partition of $[s^2, (s+1)^2)$ into $s^4$ sub-intervals of equal length. That is, $t_0 = s^2$, $t_{s^4} = (s+1)^2$ and $t_{k+1} - t_k = (2s+1)/s^4 \leq 3s^{-3}$. The left hand side in~\eqref{eqn:102} can be bounded above by
\begin{equation}
\begin{split}
	\bbP & 
		\big( \exists k \in {0, \dots, s^4} \,:\,\, \xi^{0, \bbT_0^+}_{t_k} \not \ni 0 \big) \\
		& + \bbP \Big( \exists k \in {0, \dots, s^4 -1 } : \,\,
			\sum_{x \sim 0} \cN_{x,0}^\circ \big( [t_k, t_{k+1}) \big) 
				+ \cN_{x,0}^\bullet \big( [t_k, t_{k+1}) \big) \geq 2 \Big), 	
\end{split}
\end{equation}
where the second term is a bound on the probability that a site is infected and then healed during any time interval $[t_k,t_{k+1}]$. Using the Union Bound (and the tail of the Poisson distribution), this second term is bounded above by
\begin{equation}
\label{eqn:104}
	C s^4 (t_{k+1} - t_k)^2 \leq C' s^{-2} \,.
\end{equation}
The first term can be bounded above by 
\begin{equation}
\label{eqn:103}
	s^4 \sup_{t \in [s^2, (s+1)^2)} \bbP(\xi^{0, \bbT_0^+}_t \not \ni 0 \big) \,,
\end{equation}
and it remains to bound $\bbP \big(\xi^{0, \bbT_0^+}_t \not \ni 0 \big)$.

Let $y$ be the parent of $\bbT_0$. By the duality relation~\eqref{eqn:203} and Lemma~\ref{lem:hitting}, for all $t \geq 0$ large enough
\begin{equation}
	\bbP(\xi^{0, \bbT_0^+}_t \ni 0 \big) 
	\, = \, \bbP \big( \hat{\xi}_t^{0, \bbT_0^+} \ni 0 \ \text{ or }\ 
		\hat{\tau}^{0, \bbT_0^+}_y \leq t \big)
	\, \geq \, \bbP \big( \hat{\tau}^{0, \bbT}_y \leq t \big)
	\, \geq \, 1 - C t^{-3} \,.
\end{equation}
Therefore~\eqref{eqn:103} is bounded above by $C s^{-2}$. Combining this with~\eqref{eqn:104} we see that \eqref{eqn:102} holds as desired. This completes the proof of the Theorem.
\end{proof}

In fact, the following lemma, which is required for the proofs of the corollaries, shows that above $\lambda_c(\bbT)$, global and complete survival are equivalent up to an event with zero probability.
\begin{lem}
\label{lem:CompleteOnGlobal}
Let $d  \geq 3$ and $\lambda > \lambda_c(\bbT)$. Then for any finite $\xi_0 \in \cX$,
\begin{equation}
	\bbP \big(\Omega^{\xi_0}_c \big) = \bbP \big( \Omega^{\xi_0}_g \big) \,.
\end{equation}
\end{lem}

\begin{proof}
Recall that definition~\eqref{eqn:3} of $\Omega^{\xi_0}_c$ does not depend on the observed vertex $x \in \bbT$ and we can therefore choose $x=0$. For $s \geq 0$ and $u \geq 0$ define
\begin{equation}	
	\cA^{\xi_0}(s,u) := \big \{ \xi^{\xi_0}_t \ni 0 \text{ for all } t \in [s, s+u] \big\} \, .
\end{equation}
Since $\Omega^{\xi_0} = \cup_{s \geq 0} \cap_{u \geq 0} \cA^{\xi_0}(s,u)$ we have
\begin{equation}
	\lim_{s \to \infty} \lim_{u \to \infty} \bbP \big( \cA^{\xi_0}(s,u) \big) 
		= \bbP \big( \Omega^{\xi_0}_c \big) \,.
\end{equation}
Now, we claim that as $r \uparrow \infty$,
\begin{equation}
\label{eqn:43}
	\bbP \big( \Omega_c^{\rmB_0(r)} \big) \uparrow \alpha = 1 \,.
\end{equation}
Indeed, by monotonicity the limit exists and so we may write
\begin{equation}
\begin{split}
	\bbP \big( \Omega_c^{\xi_0} \big) 
		& \, = \,\lim_{s \to \infty} \lim_{u \to \infty}  
			\bbP \big( \Omega_c^{\xi_0} \cap \cA^{\xi_0}(s,u) \big) \\
		& \, = \, \lim_{s \to \infty} \lim_{u \to \infty}  
			\bbP \big( \cA^{\xi_0}(s,u) \big) 
			\bbE \Big[ \bbP \big( \Omega_c^{\xi_0} \big| \xi_{s+u}^{\xi_0} \big)
				\Big| \cA^{\xi_0}(s,u) \Big]
		\ \leq \ \bbP \big( \Omega_c^{\xi_0} \big) \alpha \,.
\end{split}
\end{equation}
The last inequality follows from monotonicity again, since $\xi_{s+u}^{\xi_0}$ must be included in some $\rmB_0(r)$ for $r$ large enough. Since $\bbP \big( \Omega_c^{\xi_0} \big) > 0$ it follows that $\alpha$ must be $1$.

Now if $\lambda > \lambda_l(\bbT)$ then for any vertex $x \in \bbT$ and $r \geq 0$, there exists $s_{x,r} < \infty$ such that
\begin{equation}
	\bbP \big( \tau_{\rmB_0(r)}^x < s_{x,r} \big) 
		\geq \tfrac12 \bbP \big( \Omega_l^x ) > \epsilon \,.
\end{equation}
for some $\epsilon > 0$ independent of $x$ or $r$. Since on $\Omega_g^{\xi_0}$ infected vertices exist at all times, it follows from monotonicity and Markov property that for all $r \geq 0$,
\begin{equation}
	\bbP \big( \Omega_g^{\xi_0} \cap \{\tau_{\rmB_0(r)}^{\xi_0} = \infty\} \big) = 0 \,.
\end{equation}
Consequently we may write,
\begin{equation}
	\bbP \big(\Omega_c^{\xi_0} \big| \Omega_g^{\xi_0} \big) 
		\, \geq \, \bbP \big(\Omega_c^{\xi_0} \big| \tau_{\rmB_0(r)}^{\xi_0} < \infty \big) 
			\bbP \big( \tau_{\rmB_0(r)}^{\xi_0} < \infty  \, \big| \, \Omega_g^{\xi_0} \big) 
		\, \geq \, \bbP \big( \Omega_c^{\rmB_0(r)} \big) \,.
\end{equation}
Taking $r \to \infty$ and using~\eqref{eqn:43} we get $\bbP \big(\Omega_c^{\xi_0} \big| \Omega_g^{\xi_0} \big) = 1$, as desired.
\end{proof}

\begin{proof}[Proof of Corollary~\ref{cor:CompelteConvergenceXi}]
This is an immediate consequence from Theorem~\ref{thm:complete} and Lemma~\ref{lem:CompleteOnGlobal}. Indeed once $\lambda > \lambda_l(\bbT)$ we have for all $\xi_0 \in \cX$,
\begin{equation}
\label{eqn:401}
	\lim_{t \to \infty} \xi_t^{\xi_0} = 1_{\Omega_g^{\xi_0}} \bbT \, + \,
		(1-1_{\Omega_g^{\xi_0}}) \emptyset \qquad \bbP \text{-almost surely,}
\end{equation}
where we recall that the topology in $\cX$, viewed as the space of functions on $\bbT$, 
is that of pointwise convergence. This immediately gives~\eqref{e:2} and shows that any invariant measure must be a convex combination of $\delta_\bbT$ and $\delta_\emptyset$.
\end{proof}

\begin{proof}[Proof of Corollary~\ref{cor:Fixation}]
Part~\ref{item:CS1} follows immediately from~\eqref{eqn:401}. Part~\ref{item:CS2} holds because 
once $\lambda < \lambda_l(\bbT)$, starting from a finite configuration, the infection either dies out, or survives globally but not locally. In both cases, every vertex will eventually become $-$ and therefore fixate.
\end{proof}

\begin{proof}[Proof of Theorem~\ref{thm:CC_BCRW}]
Any of the duality relations can be used to prove this theorem. 
When $\lambda > \lambda_l(\bbT)$, Corollary~\ref{cor:CompelteConvergenceXi}, relation~\eqref{eqn:201} and~\eqref{eqn:2} imply that for all non-empty $\hat{\xi}_0$ and any $A \subseteq \bbT$ as $t \to \infty$,
\begin{equation}
	\bbP \big( \hat{\xi}_t^{\hat{\xi}_0} \cap A \neq \emptyset \big) 
	 	\, = \, \bbP \big( \xi_t^A \cap \hat{\xi}_0 \neq \emptyset \big)  
		\to \bbP \big( \Omega_g^A \big) = 1-\lambda^{-|A|} \,.
\end{equation}
This shows~\eqref{e:3}. On the other hand, if $\lambda \in \big[1, \lambda_l(\bbT) \big)$ then
for all finite $A \subseteq \bbT$ and finite $\hat{\xi}_0 \in \cX$ the above becomes,
\begin{equation}
	\bbP \big( \hat{\xi}_t^{\hat{\xi}_0} \cap A \neq \emptyset \big) 
	 	\, = \, \bbP \big( \xi_t^A \cap \hat{\xi}_0 \neq \emptyset \big)  
		\to 0 \,.
\end{equation}
as $t \to \infty$. This shows~\eqref{e:4}.
\end{proof}

\subsection{Proof of Lemma~\ref{lem:Zhang}}
\label{sub:ProofOfZhang}
The proof will be carried out using a number of lemmas. For an infinite connected bounded-degree graph $\bbG = (V,E)$, possibly with associated boundary conditions, we shall write $\lambda_g(\bbG)$ and $\lambda_l(\bbG)$ for the threshold value of $\lambda$ for the possibility of global and local survival for $\xi_\cdot$ when the underlying graph is $\bbG$. Formally,
\begin{equation}
\begin{split}
	\lambda_g(\bbG) & := \inf \big\{\lambda > 0 :\: \bbP(\Omega^{\xi_0, \bbG}_g) > 0 \big\} \,, \\
	\lambda_l(\bbG) & := \inf \big\{\lambda > 0 :\: \bbP(\Omega^{\xi_0, \bbG}_l) > 0 \big\} \,,
\end{split}
\end{equation}
where $\Omega^{\xi_0, \bbG}_g$ and $\Omega^{\xi_0, \bbG}_l$ are defined as in \eqref{eqn:301} and \eqref{eqn:302} with $\bbG$ being the underlying graph and $\xi_0$ is any finite non-empty initial configuration. Notice that as $\bbG$ may have associated boundary conditions, it is no longer clear that $\lambda_g(\bbG) = 1$. 

Recall that $\bbT_0$ is the sub-tree of $\bbT$ rooted at $0$. Our first lemma shows that if a graph $\bbG$ (with or without boundary conditions) contains a copy of this sub-tree which is accessible only through its root, then its threshold values are at least as small as those of $\bbT$. Note by the monotonicity statement (\ref{monotonicity}) it is enough to show this for $\bbG=\bbT_0^-$.
\begin{lem}
\label{lem:T0Thresholds}
Fix $d \geq 3$. If $\bbG$ is any infinite connected bounded-degree graph, possibly with associated boundary conditions, that contains a copy of $\bbT_0$, which is connected to the rest of the graph only through its root $0$, then 
\begin{eqnarray}
\label{eqn:20}
	\lambda_g(\bbG) & \leq 	& \lambda_g(\bbT) = 1\,, \\
\label{eqn:20.1}
	\lambda_l(\bbG) & \leq 	& \lambda_l(\bbT) \,.
\end{eqnarray}
\end{lem}

\begin{proof}
By monotonicity it is enough to show this for $\bbG=\bbT_0^-$. 
Fix any $\lambda$ such that
\begin{equation}
\label{eqn:21}
	\bbP(\Omega_g^{0, \bbT} \setminus \Omega_l^{0, \bbT}) > 0.
\end{equation}
This is always possible, since $\lambda_g(\bbT) < \lambda_l(\bbT)$ for all $d \geq 3$,
as shown in part~\ref{item:CBOL1} of Proposition~\ref{prop:Bounds}. Clearly, the distribution of $\xi^{0, \bbT}_{\cdot}$ on $\bbT_x$ is the same for any neighbor $x$ of $0$. Also, at any time $t$ we have $\big| \xi_t^{0, \bbT} \big| < \infty$. These two facts, along with~\eqref{eqn:21}, imply that we may find $T > 0$ and a finite subset of vertices $A \subseteq \bbT_x \setminus \{x\}$ such that the following event has positive probability:
\begin{equation}
\{ \xi_T^{0, \bbT} \cap \bbT_x = A \, \text{ and } \,
   	\xi_t^{0, \bbT} \cap \bbT_x \neq \emptyset \,,\,\, x \notin \xi_t^{0, \bbT} 
   	 \, \text{ for all } t \geq T \}.
\end{equation}
By the Markov property, it follows that 
\begin{equation}\label{eqn:210}
\bbP(\xi_t^{A, \bbT}\neq \emptyset\mbox{ and } x\notin\xi_t^{A, \bbT} \, \mbox{ for all }t\geq 0)>0.
\end{equation}
Observe that this probability does not change if we add $-$ boundary conditions on $\bbT \setminus \bbT_x$. Since, in addition, any two finite configurations are obtainable from each other using a finite number of transitions, we arrive to,
\begin{equation}\label{eqn:211}
\bbP(\xi_t^{x, \bbT_x^-}\neq \emptyset \mbox{ for all }t\geq 0)>0.
\end{equation}
This shows~\eqref{eqn:20} as $\bbT_x^-$ is isomorphic to $\bbT_0^-$ and $\lambda\in(\lambda_g(\bbT),\lambda_l(\bbT))$ was arbitrary.

Next, suppose that $\lambda' > \lambda_l(\bbT)$. Since $\Omega_l^{0, \bbT, \lambda'}$ occurs with positive probability, there must exist $\delta > 0$ and $T_x > 0$ for all $x \in \bbT$ such that
\begin{equation}
	\bbP(\tau_0^{x, \bbT, \lambda'} < T_x) > \delta \,,\  \forall x \in \bbT \,.
\end{equation}
This still holds, under $-$ boundary conditions, that is
\begin{equation}
\label{eqn:28}
	\bbP(\tau_0^{x, \bbT_0^-, \lambda'} < T_x) > \delta \,,\  \forall x \in \bbT_0 \,.
\end{equation}
Since $\lambda' > \lambda$, it follows from~\eqref{eqn:211} via monotonicity that 
\begin{equation}
\label{eqn:29}
	\{\xi_t^{0, \bbT_0^-, \lambda'} \neq \emptyset \; ; \;\; t \geq 0\}
\end{equation}
occurs with positive probability. But on this event, by~\eqref{eqn:28} and monotonicity, at all times $t \geq 0$ there will be a vertex $x \in \xi_t^{0, \bbT_0^-, \lambda'}$, from which there is at least $\delta$ probability of reinfecting the origin within $T_x$ time. It follows then from the Markov property that the probability of \eqref{eqn:29} and the origin being infected only finitely many times is $0$. Consequently $\bbP(\Omega_l^{0, \bbT_0^-, \lambda'}) > 0$ which implies $\lambda_l(\bbT_0^-) < \lambda'$ and since $\lambda'$ is arbitrarily close to $\lambda_l(\bbT)$, inequality~\eqref{eqn:20.1} follows.
\end{proof}

For $x \in \bbT$ and $y \in \bbT_x$ we let $\bbT_{xy} := (\bbT_x \setminus \bbT_y) \cup \{y\}$. Then we have,
\begin{lem}
\label{lem:TxTxyThresholds}
Fix $d \geq 3$. Let $x \in \bbT$ and $y \in \bbT_x$. Then 
\begin{enumerate}
\item $\lambda_g(\bbT_{x}^-) = \lambda_g(\bbT_{xy}^-) = \lambda_g(\bbT) = 1$.
\item $\lambda_l(\bbT_{x}^-) = \lambda_l(\bbT_{xy}^-) = \lambda_l(\bbT)$.
\end{enumerate}
where for the boundary conditions, both $\bbT_x$ and $\bbT_{xy}$ are treated as subgraphs of $\bbT$.
\end{lem}
\begin{proof}
Monotonicity implies that
\begin{equation}
  \lambda_g(\bbT_{xy}^-) \geq \lambda_g(\bbT_{x}^-) \geq \lambda_g(\bbT) = 1
	\quad \text{and} \quad
  \lambda_l(\bbT_{xy}^-) \geq \lambda_l(\bbT_{x}^-) \geq \lambda_l(\bbT).
\end{equation}
On the other hand, both $\bbT_{x}^-$ and $\bbT_{xy}^-$ contain a copy of $\bbT_0$ which is connected to the rest of the graph only though its root. Therefore the opposite inequalities are a consequence of Lemma~\ref{lem:T0Thresholds}.
\end{proof}

\begin{proof}[Proof of Lemma~\ref{lem:Zhang}]

Fix $\lambda > \lambda_l(\bbT)$. By the duality relation~\eqref{eqn:201} showing \eqref{eqn:13} is equivalent to showing
\begin{equation}
\label{eqn:13'}
  \inf_{t \geq 0} \bbP(0 \in \xi^{0, \bbT_{0}^-}_t) > 0 \,.
\end{equation}
We first argue that there exists $r > 0$, $s > 0$ and 
\begin{equation}
\label{eqn:30}
	p > 1/\sqrt{d-1}
\end{equation}
such that for any vertex $x$ of $0$ whose distance from $0$ is $r$ we have 
\begin{equation}
\label{eqn:BeatingSqrtD}
  \bbP(\xi^{0, \bbT_{0x}^-}_{s} \ni x) > p^r \,.
\end{equation}
Indeed, by Lemma~\ref{lem:TxTxyThresholds} we know that $\lambda > \lambda_l(\bbT_{0}^-) = \lambda_l(\bbT)$. Therefore, we may find $p' > 1/\sqrt{d-1}$ and integer $r' > 0$ large enough such that
\begin{equation}
  (p')^{r'} < \tfrac12 \bbP \big(\Omega_l^{0, \bbT_{0}^-} \big) \,.
\end{equation}
Enumerating the vertices on some path going down from $0$ as 
\begin{equation}
\label{eqn:33}
	0 = x_0, x_1, x_2, \dots \,, 
\end{equation}
there exists $s' > 0$ such that
\begin{equation}
	\bbP \big(\tau^{0, \bbT_{0}^-}_{x_{r'}} \leq s' \big) 
	  > (p')^{r'} \,.
\end{equation}
Using monotonicity and the Markov property we may iterate the above to get for all $k \geq 1$
\begin{equation}
	\bbP \big(\tau^{0, \bbT_{0 x_{k r'}}^-}_{x_{k r'}} \leq k s' \big) =
	\bbP \big(\tau^{0, \bbT_{0}^-}_{x_{k r'}} \leq k s' \big) 
	  > (p')^{k r'} \,.
\end{equation}
Now, write
\begin{equation}
\begin{split}
(p')^{kr'} e^{-d} 
	& <
  		\bbP \big(\tau^{0, \bbT_{0 x_{k r'}}^-}_{x_{k r'}} \leq k s' 
  		\text{ and } \xi_t^{0, \bbT_{0x_{kr'}}^-} \ni x
  		\text{ for all } t \in [\tau^{0, \bbT_{0 x_{k r'}}^-}_{x_{k r'}} ,\,
  			\tau^{0, \bbT_{0 x_{k r'}}^-}_{x_{k r'}} + 1] \big) \\
  	& \leq 
  		\sum_{t=0,1,\dots}^{\lfloor ks' + 1 \rfloor} 
  			\bbP(\xi_t^{0, \bbT_{0 x_{kr'}}^-} \ni x_{k r'}) 
  	   \, \leq \, (ks'+2) \bbP(\xi_{s''}^{0, \bbT_{0 x_{kr'}}^-} \ni x_{k r'})\,,
\end{split}
\end{equation}
where $s''$ is the index of the maximal term in the last sum. Therefore (for $k$ large enough),
\begin{equation}
	\bbP (\xi_{s''}^{0, \bbT_{0 x_{kr'}}^-} \ni x_{k r'})
		> C/(k s') (p')^{kr'}\,,
\end{equation}
which implies that \eqref{eqn:BeatingSqrtD} holds with $r := kr'$, $s: = s''$, $x = x_r$ 
and some $p \in (1/\sqrt{d-1},\, p')$, once we choose $k$ large enough. Finally, since the choice of path in \eqref{eqn:33} is arbitrary, $x_r$ can be replaced with any vertex whose distance from $0$ is $r$. 

Next we introduce a modified version of $\xi_\cdot^{0, \bbT_{0}^-}$ which we denote by $\xi'_\cdot$. The process $\xi'_\cdot$ is still Markovian and takes values in the space of all configurations on $\bbT_0$. It starts from a single infection at the origin and evolves exactly as $\xi_\cdot^{0, \bbT_{0}^-}$ does, only that at times $t=ks$ where $k=1, \dots$, we heal all vertices whose distance from the origin is greater than $kr$ and we heal and keep healed the ones whose distance from the origin is less than $kr$. Formally, we set
\begin{equation}
\begin{split}
	\xi'_{ks}(x) & = -  \quad \text{for } x \text{ s.t. } \rho(0,x) > kr \,, \\
	\xi'_t(x) & = -  \quad \text{for } x \text{ s.t. }\rho(0,x) < kr \text{ and all } t \geq ks \,.
\end{split}
\end{equation}
By monotonicity $\xi^{0, \bbT_0^-}_t$ stochastically dominates $\xi'_t$ for all $t$. At the same time, it is easy to see that the process $Z_k := |\xi'_{t_k}|$ where $k=0, \dots, $ is a branching process with mean reproduction $\mu := \bbE Z_1 = (d-1)^r p^r > (d-1)^{r/2} > 1$ and $\bbE (Z_1)^2 < \infty$. Therefore, there exists $\epsilon > 0$ such that for all $k$ 
\begin{equation}
\label{eqn:39}
	\bbP \Big(|\xi_{ks}^{0, \bbT_0^-} \cap \rmS_0(kr) | >  \tfrac12 ((d-1)p)^{kr} \Big)
		\geq \bbP \big( Z_k > \tfrac12 \mu^k \big) > \epsilon \,,
\end{equation}
where we recall $\rmS_x(r) := \{y \in \bbT :\: \rho(x,y) = \lfloor r \rfloor \}$ for $x \in \bbT$ and $r \geq 0$.

Now for $k$ large enough such that the above holds and whose precise value will be chosen later, 
set $\tilde{s}:= ks$, $\tilde{r} := kr$ and $\tilde{p} := p^{\tilde{r}}$. We now prove by induction that for all positive and even $n \in \bbN$
\begin{equation}
\label{eqn:40}
	p_n := \bbP \big( \xi_{n\tilde{s}}^{0, \bbT_0^-} \ni 0 \big) > \tfrac12 \epsilon 
		\tilde{p} \,.
\end{equation}
For $n=2$, iterating~\eqref{eqn:BeatingSqrtD} $k$ times along the path from
$x \in \rmS_0(\tilde{r})$ to $0$ and using~\eqref{eqn:39},
\begin{equation}
\label{eqn:41}
  p_2 \geq \bbP \Big( |\xi_{\tilde{s}}^{0, \bbT_0^-} \cap \rmS_0(\tilde{r}) | \geq 1 \Big)
  	\bbP \Big( \xi_{2\tilde{s}}^{0, \bbT_0^-} \ni 0 \,\, \Big| \,\, 
  		|\xi_{\tilde{s}}^{0, \bbT_0^-} \cap \rmS_0(\tilde{r}) | \geq 1 	\Big)
  	\geq \epsilon \tilde{p} 
\end{equation}
For $n+2 > 2$, we can bound $p_{n+2}$ below by
\begin{equation}
\label{eqn:42}
\begin{split}
	\bbP \Big( |\xi_{\tilde{s}}^{0, \bbT_0^-} & \cap \rmS_0(\tilde{r}) |
		>  \tfrac12 ((d-1)p)^{\tilde{r}} \Big) \\
	& \times 
		\bbP \Big( |\xi_{(n+1) \tilde{s}}^{0, \bbT_0^-} \cap \rmS_0(\tilde{r}) | \geq 1 
		\, \Big | \,
		|\xi_{\tilde{s}}^{0, \bbT_0^-} \cap \rmS_0(\tilde{r}) | > \tfrac12 ((d-1)p)^{\tilde{r}} \Big) \\
	& \times
			\bbP \Big( \xi_{(n+2) \tilde{s}}^{0, \bbT_0^-} \ni 0
			\, \Big | \,
			|\xi_{(n+1)\tilde{s}}^{0, \bbT_0^-} \cap \rmS_0(\tilde{r}) | \geq 1 ,\,
			|\xi_{\tilde{s}}^{0, \bbT_0^-} \cap \rmS_0(\tilde{r}) | > \tfrac12 ((d-1)p)^{\tilde{r}} \Big) \,.
\end{split}
\end{equation}
We can bound below the first term by $\epsilon$ using~\eqref{eqn:39} and the last term by $\tilde{p}$ using the argument in~\eqref{eqn:41}. For the middle one, we use the fact that for any configuration $\eta$ on $\bbT_0$
\begin{equation}
	\xi_t^{\eta, \bbT_0^-} \geq 
		\bigvee_{x \in \rmS_0(\tilde{r})} \xi_t^{\eta_x, \bbT_x^-} \mbox{ for all }t\geq 0,
\end{equation}
where $\big( \xi^{\eta_x, \bbT_x^-}_\cdot :\: x \in \rmS_0(\tilde{r}) \big)$ are independent, $\eta_x$ is the restriction of $\eta$ to $\bbT_x$. This is because of monotonicity~\eqref{monotonicity}, the tree-structure and the choice of boundary conditions. Using this and the induction hypothesis for $p_n$, we can bound below the second term in~\eqref{eqn:42} by 
\begin{equation}
	1 - (1-p_n)^{\tfrac12 ((d-1)p)^{\tilde{r}}} 
	\geq 1 - \exp \big(-\tfrac14 \epsilon ((d-1)p^2)^{kr} \big) \,.
\end{equation}
In light of~\eqref{eqn:30}, by choosing $k$ large enough we can guarantee that the right hand above is at least $\tfrac12$ and conclude that~\eqref{eqn:41} holds for $p_{n+2}$ as well.

Once we have~\eqref{eqn:41} for all even positive $n$, it is only a short step to complete the proof of the lemma. Indeed, for any $t \geq 0$, find a positive even $n$ such that
$n\tilde{s} \leq t < (n+2) \tilde{s}$ and write
\begin{equation}
	\bbP \bigr(\xi_t^{0, \bbT_0^-} \ni 0 \bigl) \, \geq \,
		\bbP \bigl(\xi_{n\tilde{s}}^{0, \bbT_0^-} \ni 0 \bigr) \,
			\bbP \bigl(\cN_{x,0}^\circ([n\tilde{s}, t]) = 0 \text{ for all } x \sim 0 \bigr)
		> \tfrac12 \epsilon \tilde{p} e^{-d} > 0 \,.
\end{equation}
This shows~\eqref{eqn:13'} and completes the proof.
\end{proof}

\subsection{Proof of Lemma~\ref{lem:hitting}}
\label{sub:Hitting}
The proof will consist of a sequence of lemmas.

\begin{lem}
\label{lem:DistanceBound}
Fix $d \geq 3$ and $\lambda \geq 1$. For all $\delta_1 > 0$ there exists $b > 0$ such that for all $t \geq 0$ large enough.
\begin{equation}
	\bbP (\hat{\xi}_t^{0, \bbT} \subseteq \rmB_0(bt) ) \geq 1-e^{-\delta_1 t} \,.
\end{equation}
\end{lem}
\begin{proof}
The proof will follow by coupling of $\hat{\xi}_t^{0, \bbT}$ with a (continuous time) branching random walk on $\bbR$, whose growth rate is well controlled. (Alternatively, one can use 
a comparison to last passage percolation on $\bbT$, or just prove this via elementary methods). To this end, we first introduce the following variant of $\hat{\xi}_\cdot^{0, \bbT}$ which we denote by $(\tilde{\xi}_t^{0, \bbT} :\: t\geq 0)$. The process $\tilde{\xi}_\cdot^{0, \bbT}$ starts from as single particle at $0$ and evolves as $\hat{\xi_\cdot}$ does, except for two differences. First, there are no coalescences, that is more than one particle can share a single vertex. Second, whenever a particle at vertex $v$ moves to (rate $1$) or produces a particle at (rate $\lambda - 1$) its parent $u \sim v$, one of its children $w \sim v$ are chosen (according to some fixed method) instead of $u$. 

If $\hat{R}^{0, \bbT}_t$, $\tilde{R}^{0, \bbT}_t$, denote the maximal graph-distance of a particle in $\hat{\xi}_t^{0, \bbT}$, resp. $\tilde{\xi}_t^{0, \bbT}$ from the origin, then by a straightforward coupling,
\begin{equation}
	\hat{R}^{0,\bbT}_t \leq_s \tilde{R}^{0,\bbT}_t \,.
\end{equation}
At the same time, the process $(N_t :\: t \geq 0)$ defined as
\begin{equation}
	N_t := \sum_{x \in \tilde{\xi}_t^{0, \bbT}} \delta_{\rho(0, x)} \,,
\end{equation}
is a continuous time branching random walk on $\bbR_+$ with $N_0 = \delta_0$ and whose reproduction measure on $\bbR$ is $\delta_1$ at rate $d$ and $\delta_1 + \delta_0$ at rate $(\lambda - 1)d$. (That is, a particle at displacement $r \geq 0$ is replaced by a particle at displacement $r+1$ at rate $d$ and by two particles: at $r$ and $r+1$, at rate $(\lambda-1)d$.)

Writing
\begin{equation}
\begin{split}
	\bbP(\hat{\xi}_t^{0, \bbT} \subseteq \rmB_0(bt)) & 
		= \bbP(\hat{R}^{0, \bbT}_t \leq bt) \geq \bbP(\tilde{R}^{0, \bbT}_t \leq bt) \\
		& = 1 - \bbP(N_t(bt, \infty) \geq 1) \geq 1 - \bbE(N_t(bt, \infty)) \,.
\end{split}
\end{equation}
Now Theorem 4 in \cite{biggins1995growth} says (note that  ``non-lattice'' there refers to the distribution of times between reproductions, not the support of the reproduction measures) 
\begin{equation}
	\frac{\log \bbE N_t[bt, \infty))}{t} \to \alpha^*(b) \,.
\end{equation}
where $\alpha^*(b)$ depends on the Laplace transform of the reproduction measures (an analog of the Legendre transform in Cramer's theorem) and in the case of reproduction measures with finite support and exponential reproduction times, can be made arbitrarily small by choosing $b$ large enough. 
\end{proof}

\begin{lem}
\label{lem:sizeofdual}
Let $d \geq 3$ and $\lambda > 1$. There exist $a > 0$, $\delta_2 > 0$ such that for all $t \geq 0$ large enough
\begin{equation}
\label{eq:growthrate}
\bbP \big( |\hat{\xi}^{0, \bbT}_t| \geq e^{at} \big) \geq 1-e^{-\delta_2 t}  \,.
\end{equation}
\end{lem}
\begin{proof}
We shall omit the superscript $\bbT$ as all processes in this proof run on $\bbT$. By the thinning relationship \eqref{eqn:221} with initial state $\delta_0$, we can write for $p = \bbP(\Omega_g^0)$ and any $a > 0$,
\begin{equation}
\label{eq:thinsize}
	\bbP(|(\xi^0_t)^{(p)}| \geq e^{at}) 
		= \bbP (|\hat{\xi}^{(\delta_0^{(p)})}_t| \geq e^{at})
		= p \bbP (|\hat{\xi}^0_t| \geq e^{at}) \,.
\end{equation}

Since every infected vertex in $\xi^0_t$ stays infected in $(\xi^0_t)^{(p)}$ with probability $p$ independently of other vertices, it follows from Cramer's theorem applied to this sequence of Bernoulli$(p)$ random variables that
\begin{equation}
\begin{split}
	\bbP(|(\xi^0_t)^{(p)}| \geq e^{at}) 
		\, & \geq \, \bbP \big( |\xi^0_t| \geq \tfrac{2}{p} e^{at} \big) \,
			\bbP \big( |(\xi^0_t)^{(p)}| \geq e^{at} \big| |\xi^0_t| \geq \tfrac{2}{p} e^{at} \big)	 \\
		   & \geq \bbP \big( |\xi^0_t| \geq e^{a't} \big) (1 - e^{-C'e^{Cat}})  \,
\end{split}
\end{equation}
where $a' = 2a$  and $t \geq 0$ is large enough. Therefore it is enough to show that for some $\delta > 0$ and all large $t$,
\begin{equation}
\label{eqn:65}
	p^{-1} \bbP \big( |\xi^0_t| \geq e^{a't} \big) \geq 
		\bbP \big( |\xi^0_t| \geq e^{a't} \, \big| \Omega_g^0 \big) \geq 1 - e^{-\delta t} \,.
\end{equation}

As discussed in the introduction (see discussion above~\eqref{eqn:303}), the transitions of $|\xi^0_\cdot|$ are those of a nearest neighbor random walk with drift $\frac{\lambda - 1}{\lambda+1}$ and an absorbing state at $0$. These transitions occur at rate 
\begin{equation}
	(\lambda + 1) |\partial_{\bbT} \xi^0_t| \geq (\lambda + 1) |\xi^0_t| \,,
\end{equation}
where we recall that $\partial_{\bbT} \xi^0_t$ denotes the set of edges of $\bbT$ with exactly one vertex in $\xi^0_t$ and the last inequality holds since $|\partial_\bbT A| \geq (d-2) |A|$ for any finite $A \subset \bbT$.

It follows that we can couple $(|\xi^0_t| :\: t\geq 0)$ with a continuous time birth-and-death process $(Y_t :\: t \geq 0)$ on $\bbN$ with birth rates $p(y) = \lambda y$ and death rates $q(y) = y$ such that both processes start from $1$, make the same transitions and that times between successive transitions of $|\xi^0_\cdot|$ are less or equal than the corresponding ones of $Y_\cdot$. Thus, if we define
\begin{equation}
T := \inf \{t \geq 0 :\: |\xi^0_t| \geq 2e^{a' t} \} 
\text{ and }
S := \inf \{t \geq 0 :\: Y_t \geq 2e^{a' t} \} 
\end{equation}
it follows from this coupling that
\begin{equation}
\label{eqn:60}
	\bbP(T \leq t \,|\, \Omega^0_g) \geq \bbP(S \leq t \,|\, Y_t > 0 ,\,t \geq 0) \,.
\end{equation}
For $Y_t$, either explicit calculation or e.g. \cite{BDchain} shows that there exist $a' > 0$ and $\delta' > 0$ such that
\begin{equation}
\label{eqn:61}
	\bbP( S \leq t \,|\, Y_t > 0 ,\,t \geq 0) \, \geq \,
		\bbP(Y_t \geq 2e^{a' t} \,|\, Y_t > 0 ,\,t \geq 0)
			\geq 1 - e^{-\delta' t} \,.
\end{equation}
At the same time, 
\begin{equation}
\begin{split}
\label{eqn:62}
	\bbP \big(|\xi^0_t| > e^{a't} \,\big|\, \Omega^0_g, T \leq t \big)
	& \geq \bbP(|\xi^0_{T+s}| - |\xi^0_T| > -e^{a't} ,\, s \geq 0 \, \big| \, \Omega^0_g, T \leq t) \\
	& \geq \bbP(|\xi^0_{T+s}| - |\xi^0_T| > -e^{a't} ,\, s \geq 0 \, \big| \, T \leq t) 
	= 1-\lambda^{-e^{a't}} \,,
\end{split}
\end{equation}
where $1-\lambda^{-e^{a't}}$ is the standard gambler-ruin probability.

Combining~\eqref{eqn:60},~\eqref{eqn:61} and~\eqref{eqn:62} we arrive to
\begin{equation}
	\bbP \big(|\xi^0_t| > e^{a't} \,\big|\, \Omega^0_g \big) 
		\geq 1 - e^{-\delta' t} - \lambda^{-e^{a't}} \geq 1 - e^{-\delta t} \,,
\end{equation}
for a suitable $0 < \delta < \delta'$, as required in~\eqref{eqn:65}. The result follows.
\end{proof}

As an immediate consequence of Lemma~\ref{lem:DistanceBound} and Lemma~\ref{lem:sizeofdual} we get
\begin{lem}
\label{lem:Congestion}
Fix $d \geq 3$ and $\lambda > 1$. There exists $\delta > 0$, $a > 0$, $b > 0$ such that for all $t \geq 0$ large enough
\begin{equation}
\label{eqn:47}
	\bbP \big( \big| \hat{\xi}_t^{0, \bbT} \cap \rmB_0(bt) \big| \geq e^{at} \big)
		\geq 1 - e^{-\delta t} \,.
\end{equation}
\end{lem}
\begin{proof}
Use Lemma~\ref{lem:DistanceBound} and Lemma~\ref{lem:sizeofdual} and the union bound.
\end{proof}

The next lemma shows that when $\lambda > \lambda_l(\bbT)$, starting from a single occupied vertex $x$, any neighboring vertex $y\sim x$ will eventually become occupied, with high probability, even if we restrict the underlying graph to a finite sub-set of $\bbT$, but as long as this sub-graph is large enough.
\begin{lem}
\label{lem:Commute}
Fix $d \geq 3$ and $\lambda > \lambda_l(\bbT)$. Let $x$ and $y$ be neighboring vertices in $\bbT$.
For all $\beta > 0$, there exists $r > 0$, $u > 0$ such that
\begin{equation}
\label{eqn:48}
	\bbP \big ( \hat{\tau}_y^{x, \rmB_x^-(r)} \leq u \big) \geq e^{-\beta} \,,
\end{equation}
where $\rmB_x(r)$ is the ball of radius $r$ (in the graph-distance) around $x$, viewed as a sub-graph of $\bbT$.
\end{lem}

\begin{proof}
Fix $\lambda > \lambda_l(\bbT)$. Without loss of generality, we can assume that $x = 0$ and that $y$ is the parent of $0$. We first show that 
\begin{equation}
\label{eqn:66}
	\bbP \big( \hat{\tau}_y^{0, \bbT} < \infty \big) = 1 \,.
\end{equation}
Indeed, since $\lambda > \lambda_l(\bbT)$, local survival and the duality relation imply that
for all $z \in \bbT$, there exists $T_z < \infty$ such that
\begin{equation}
	\bbP \big( \hat{\tau}_y^{z, \bbT} < T_z \big) \geq \tfrac12 \bbP(\Omega_l^{z, \bbT})
		=: \epsilon \,.
\end{equation}
where $\epsilon > 0$ is independent of $z$. Since $\hat{\xi}_\cdot^{0, \bbT}$ never dies, at any time $t \geq 0$, there will be at least one occupied vertex $z$, which (by monotonicity and the Markov property) will give rise to a particle at $y$ within $T_z$ time with probability at least $\epsilon$. It follows that the probability of $y$ never being occupied is $0$, which is what we need for~\eqref{eqn:66}.

Now since at all times $t$ we have $|\hat{\xi}^{0, \bbT}_t| < \infty$
\begin{equation}
	\big \{ \hat{\tau}_y^{0, \bbT} < \infty \big \} =
		\cup_{r=1}^\infty \cup_{u=1}^\infty 
			\{ \hat{\tau}_y^{0, \bbT} \leq u \,,\,\, \hat{\xi}_t \subseteq \rmB_0(r) \;;\; t \in [0,
				\hat{\tau}_y^{0, \bbT}] \big \}
\end{equation}
and as the sequence of events on the right hand side is monotone increasing in $(r,u)$ it follows that
\begin{equation}
	\lim_{r \to \infty} \lim_{u \to \infty}
			\bbP \big ( \hat{\tau}_y^{0, \bbT} \leq u \,,\,\, \hat{\xi}_t \subseteq \rmB_0(r) 
				\;;\; t \in [0,\hat{\tau}_y^{0, \bbT}] \big ) = 1 \,.
\end{equation}
As a consequence we get that for any $\beta > 0$ there exist $u > 0$, $r > 0$ large enough such that
\begin{equation}
	\bbP \big ( \hat{\tau}_y^{0, \bbT} \leq u \,,\,\, \hat{\xi}_t \subseteq \rmB_0(r-1) 
		\;;\; t \in [0,\hat{\tau}_y^{0, \bbT}] \big ) \geq e^{-\beta} \,,
\end{equation}
but the above event is equivalent to that in~\eqref{eqn:48} for $x=0$.
\end{proof}

\begin{proof}[Proof of Lemma~\ref{lem:hitting}]
Without loss of generality, we can assume that $x = 0$ and that $y$ is the parent of $0$ in $\bbT$.
Let $\alpha > 0$ be arbitrarily large and $\delta$, $a$, $b$ be given by Lemma~\ref{lem:Congestion}. Setting $a' := \alpha a \wedge 1$, $b' := \alpha b$, Lemma~\ref{lem:Congestion} implies that 
for all $s \geq 0$ large enough, the event
\begin{equation}
	\hat{\cA} := \big\{ \big| \hat{\xi}_{\alpha s}^{0, \bbT} \cap \rmB_0(b' s) \big| \geq e^{a' s} \big\}
\end{equation}
satisfies
\begin{equation}
\label{eqn:50}
	\bbP(\hat{\cA}) \geq 1-e^{-\alpha \delta s} \, .
\end{equation}

For any vertex $z \in \bbT$, let $\gamma_{z}$ denote the set of vertices on the unique path from $z$ to $y$. For $r > 0$ let $\Gamma_{z}(r) := \cup_{w \in \gamma_{z}} \rmB_w(r)$. With\begin{equation}
\label{eqn:51}
	\beta := \frac{a'}{5b'} \,.
\end{equation}
where $a'$ and $b'$ are as above and with $r$, $u$ given by Lemma~\ref{lem:Commute}, set
\begin{equation}
	\hat{\cB}^z := \big \{ \hat{\tau}_y^{z, \Gamma_{z}^-(r)} \leq 2 b'us \big \} \,.
\end{equation}
By iterating Lemma~\ref{lem:Commute}, if $z \in \rmB_0(b's)$ we have
\begin{equation}
\label{eqn:53}
	\bbP(\hat{\cB}^z) \geq 
	\bbP \big(\hat{\tau}_y^{z, \Gamma_{z}^-(r)} \leq (b's+1)u \big)
	\geq e^{-\beta (b's+1) } \geq e^{-2\beta b's} \,.
\end{equation}

Now suppose that $\hat{\cA}_t$ occurs and pick $z_{0} \in \hat{\xi}_{\alpha s}^{0, \bbT} \cap \rmB_0(b's)$. Henceforth, we shall assume that there is some fixed order among all vertices of $\bbT$ and that every time we arbitrarily pick a vertex from a subset of $\bbT$ we pick the minimal one with respect to this order. Notice that if the set from which we choose is random, the chosen vertex, e.g. $z_0$ above, is a random variable. 

For what is coming, it will be useful to employ the following notation. For any $t \geq 0$, we denote $\theta_t$ the ``shifting forward of time by $t$'', that is $\theta_t$ acts on the underlying sample space, such that
\begin{equation}
\begin{split}
	\cN^{\bullet}(\cdot)_{u,v} \circ \theta_t & = \cN^{\bullet}_{u,v}(\cdot + t) \,\\
	\cN^{\circ}(\cdot)_{u,v} \circ \theta_t & = \cN^{\circ}_{u,v}(\cdot + t) \,
\end{split}
\end{equation}
for any neighboring vertices $u \sim v$ of $\bbT$.

Then using \eqref{eqn:53}, the Markov property and monotonicity we may write
\begin{equation}
	\bbP \big( \hat{\cB}^{z_{0}} \circ \theta_{\alpha s} \, \big| \, \hat{\cA} \big )
		\geq e^{-2\beta b' s}.
\end{equation}
On $\hat{\cA} \cap \hat{\cB}^{z_{0}}$ we have $\hat{\tau}^{0, \bbT}_y < \alpha s + 2b's u$ and if this indeed happens, we stop.
If not, we pick $z_1 \in \hat{\xi}_{\alpha s}^{0, \bbT} \cap \rmB_0(b's) \setminus \Gamma_{z_0}(r)$ (which must exists, as argued below). Notice that by ``removing'' (from consideration) all vertices in $\Gamma_{z_0}(r)$, conditioned on $\hat{\cA}$ and the choice of $z_1$, 
the process 
$\big(\hat{\xi}_t^{z_1, \bbT_{z_1}^-} \circ \theta_{\alpha s} \: : \: t \geq 0 \big)$
is independent of $\hat{\cB}^{z_{0}} \circ \theta_{\alpha s}$. Therefore, Lemma~\ref{lem:Zhang} guarantees that there exists a universal $\epsilon > 0$ such that
\begin{equation}
\label{eqn:55}
	\bbP \big( \hat{\xi}_{\alpha s + 2b'us}^{z_1, \bbT_{z_1}^-} \ni z_1 \, \big| \,
		\hat{\cA}, (\hat{\cB}^{z_0} \circ \theta_{\alpha s})^{\rm c} \big) \geq
	\bbP \big( \hat{\xi}_{2b'us}^{z_1, \bbT_{z_1}^-} \ni z_1 \big) \geq \epsilon \,.
\end{equation}
The event in the first term can be written as $\hat{\cC}^{z_1} \circ \theta_{\alpha s + 2b'us} $ once we set for $z \in \bbT$,
\begin{equation}
	\hat{\cC}^z := \big \{ \xi_0^{z, \bbT_z^-} \ni z \} \,.
\end{equation}
If this happens, then by monotonicity $z_1 \in \hat{\xi}_{\alpha s + 2b'us}^{0, \bbT}$ and we now check whether $\hat{\cB}^{z_1} \circ \theta_{\alpha s + 2b'us}$ occurs. Since we are only conditioning on events that depend on (measurable w.r.t.) the process up to time $\alpha s + 2b'us$, by the Markov property we still have by virtue of\eqref{eqn:53}
\begin{equation}
\label{eqn:57}
	\bbP \big ( \hat{\cB}^{z_1} \circ \theta_{\alpha s + 2b'us} \, \big| \,
		\hat{\cA}, (\hat{\cB}^{z_0} \circ \theta_{\alpha s})^{\rm c}, \, \hat{\cC}^{z_1} \circ
			\theta_{\alpha s + 2b'us} \big) \geq e^{-2\beta b' s} \,.
\end{equation}
On the intersection of all the events in~\eqref{eqn:57}, then we have 
$\hat{\tau}_y^{0, \bbT} < \alpha s + 4b'us$. If either the event in~\eqref{eqn:55} or the event in~\eqref{eqn:57} fail, we pick a new vertex $z_2 \in \hat{\xi}_{\alpha s}^{0, \bbT} \cap \rmB_0(b's) \setminus \big(\Gamma_{z_0}(r) \cup \Gamma_{z_1}(r)\big)$.

Proceeding in this fashion we obtain vertices $z_0, z_1, \dots, z_n$, for $n$ to be defined later. Indeed, vertex $z_k$ is chosen from
\begin{equation}
\label{eqn:58}
	\hat{\xi}_{\alpha s}^{0, \bbT} \cap \rmB_0(b's) \setminus \cup_{0 \leq l < k} 
		\Gamma_{z_l}(r) \,,
\end{equation}
at time $\alpha s + 2kb'us$ we have (as in~\eqref{eqn:55},~\eqref{eqn:57})
\begin{equation}
\label{eqn:59}
	\bbP \Big( (\hat{\cC}^{z_k}, \hat{\cB}^{z_k})
				 \circ \theta_{\alpha s + 2kb'us}
				\Big | \hat{\cA}, \cap_{l < k} \big( 
					(\hat{\cC}^{z_l}, \hat{\cB}^{z_l})^{\rm c}
					\circ \theta_{\alpha s + 2lb'us} \big) \Big) \geq \epsilon e^{-2\beta b' s}\,.
\end{equation}
and if the above event occurs, then $\hat{\tau}_y^{0, \bbT} < \alpha s + 2(k+1) b'us$ and we stop.
To make sure that we do not ``run out'' of vertices in~\eqref{eqn:58}, observe that on $\hat{\cA}$
when we pick $z_k$,
\begin{equation}
	\big| \hat{\xi}_{\alpha s}^{0, \bbT} \cap \rmB_0(b's) \setminus \cup_{0 \leq l < k} 
		\Gamma_{z_l}(r) \big| 
		\geq e^{a' s} - k r^d (b's + 1) \,.
\end{equation}
Thus, choosing $n:= e^{a's/2}$ guarantees that the above is positive for all $k \leq n$.

Combining~\eqref{eqn:50}, \eqref{eqn:51} and~\eqref{eqn:59}, we obtain for all $s$ large enough,\begin{equation}
\begin{split}
	\bbP \big ( \hat{\tau}_y^{0, \bbT} < e^{a's/3} \big)
	& \geq \bbP \big ( \hat{\tau}_y^{0, \bbT} < \alpha s + 2(n+1) b'us \big) \\
	& \geq 
		\bbP \Big(\cup_{0 \leq l \leq n} 
			\big( (\hat{\cC}^{z_l}, \hat{\cB}^{z_l})
				\circ \theta_{\alpha s + 2lb'us} \big) \, \big| \hat{\cA} \Big) \, \bbP(\hat{\cA})  \\
	& \geq 1 - (1 - \epsilon e^{-2\beta b' s})^{n+1} - e^{-\alpha \delta s}
		\geq 1 - e^{-\epsilon e^{(a'/2 - 2\beta b') s}} -e^{-\alpha \delta s} \\
	& \geq 1 - Ce^{-\alpha \delta s} = 1 - C \big( e^{a's/3} \big)^{\frac{-3\alpha \delta}{a'}} \,.
\end{split}	
\end{equation}
Substituting $t:=e^{a's/3}$ and $q:=3\alpha \delta/a'$ we rewrite the above as
\begin{equation}
	\bbP(\hat{\tau}_y^{0, \bbT} > t \big) \leq C t^{-q} \,,
\end{equation}
for all $t$ large enough. But since $\alpha$ was arbitrarily, $\delta$ is fixed and independent of it and $a' \leq 1$ we can make $q$ is large as we want. This shows~\eqref{eqn:14}.
\end{proof}

\section{Proof of Proposition~\ref{prop:Bounds}}
\label{sec:ProofOfBounds}
Beginning with part~\ref{item:CBOL1}, the upper bound easily follows from a comparison with the contact process. Indeed, consider the contact process $(\zeta_t^{\zeta_0} :\: t \geq 0)$ on $\bbT^d$ starting from some $\zeta_0 \in \cX$, where each site heals independently at rate $d$ and infected sites infect their neighbors at rate $\lambda$. Employing a standard coupling (for instance, using the graphical representations of both models) we can have
\begin{equation}
	\zeta_t^{\zeta_0}(x) \leq \xi_t^{\zeta_0}(x)	\quad ;\; t \geq 0,\, x \in \bbT \,.
\end{equation}
Then local survival of $\zeta^{\zeta_0}_{\cdot}$ implies local survival of $\xi_\cdot^{\zeta_0}$. By Theorem 2.2 in \cite{pemantle1992contact}, this occurs with positive probability for $\zeta^{\zeta_0}_\cdot$ as soon as $\lambda > \min \big (2d ,\, \frac{4d}{(\sqrt{d-1}-4) \vee 0} \big)$. 

Before proving the lower bound in part~\ref{item:CBOL1}, let us first address part~\ref{item:CBOL2} of the proposition, as its proof is similar and at the same time slightly simpler. For $\alpha \geq 0$ consider the function $f : \cX \to \bbR_+$ given by
\begin{equation}
	f(\xi)=\sum_{x \in \xi}\alpha^{\rho(0,x)} \,.
\end{equation}
If $\alpha < (d-1)^{-1}$ this function is uniformly bounded. At the same time, we claim that 
if $\lambda^{-1} < \alpha < \lambda$ then $M_t^{\xi_0} := f \big( \xi_t^{\xi_0} \big)$ is a sub-martingale for all finite non-empty $\xi_0 \in \cX$ (this is still true even if $\xi_0$ is infinite, but our argument below needs to be modified in this case). Indeed, let $t \geq 0$ and consider the next transition of $\xi_\cdot^{\xi_0}$ after $t$. Such transition must be due to a pair of neighboring vertices $u \sim v \in \bbT$ such that $\xi_t^{\xi_0}(u) = +$ while $\xi_t^{\xi_0}(v) = -$, for which there is either an ``arrival'' of an infection event, whereby $u$ infects $v$, or an ``arrival'' of a healing event, whereby $v$ heals $u$ and in any case this arrival is the first to occur after time $t$. 

Conditioned on $u,v$ being the vertices for which one of these arrivals occurred first, the probability that it is an infection event is $\tfrac{\lambda}{\lambda+1}$, while the probability that it is a healing event is $\tfrac{1}{\lambda+1}$. There are two cases to consider: either $\rho(0, u)=\rho(0, v)+1$ or $\rho(0, u)= \rho(0, v)-1$. The expected change in $M^{\xi_0}_t$ due to this transition is then, respectively, 
\begin{equation}
	\alpha^{\rho(0, v)}
		\left( \frac{\lambda} {1+\lambda} - \frac{\alpha}{1+\lambda} \right) \ \text{ and } \ 
	\alpha^{\rho(0, u)}
		\left( \frac{\lambda \alpha}{1+\lambda} -\frac{1}{1+\lambda} \right) \,.
\end{equation} 
Both are non negative when $\lambda^{-1} < \alpha < \lambda$ and hence $M_\cdot^{\xi_0}$ is a sub-martingale.

If $\lambda > d-1$ we may find $\lambda^{-1} < \alpha < (d-1)^{-1}$ for which $(M_
t^{\xi_0} :\: t \geq 0)$ is a bounded sub-martingale and therefore must converge to some finite value with probability 1. Consequently, for any vertex $v$ there is some time $s_v$ after which it is either always infected or never infected; otherwise there would have been an unbounded sequence of times at which $M^{\xi_0}_t$ changed by $\alpha^{\rho(0,v)}$, a contradiction to the convergence. Furthermore, two neighboring vertices cannot have different limits and hence almost-surely $\xi_t^{\xi_0}$ converges either to $\bbT$ or $\emptyset$. The latter cannot happen once $\Omega_l^{\xi_0}$ occurs, but $\Omega_l^{\xi_0}$ occurs with positive probability when $\lambda > \lambda_l(\bbT)$. Thus, when $\lambda > (d-1) \vee \lambda_l(\bbT)$ then 
with positive probability $\xi_0^{\xi_0}$ converges to $\bbT$ and the process survives completely.

We now come back to the proof of the lower bound in part~\ref{item:CBOL1}. Let $u \in \bbT$ and set $w \in \bbT$ to be the first common ancestor of $u$ and $0$. Note that it might be $0$ or $u$. Now define the ``height'' of $u$ relative to $0$ as $h(u):= \rho(w, u) - \rho(w,0)$. Next, for $\alpha > 0$ let $f(u) := \alpha^{h(u)}$ and for a subset of vertices $U \subseteq \bbT$ we set $f(U) := \sum_{u \in U} f(u)$. Notice that $f(u)$ may be $\infty$. Finally, define the process $M_t := f(\xi_t^0)$. 

We would like to claim that as soon as $\alpha$ and $\lambda$ satisfy
\begin{equation}
\label{eqn:501}
    \lambda (d-1)\alpha^2 -d\alpha+\lambda \leq 0
\end{equation}
then $(M_t :\: t\geq 0)$ is a nonnegative super-martingale and therefore must have an almost-sure finite limit $M_\infty$. As before this will imply that $\xi_t^0$ converges to either $\emptyset$ or $\bbT$ with probability $1$. Since $\bbE M_\infty \leq \bbE M_0 < \infty$, it must be that $M_\infty$ is finite almost-surely, which is only possible if the convergence of $\xi_t^0$ is to $\emptyset$. Thus, the probability of local survival is $0$ and hence $\lambda$ is a lower bound for $\lambda_l(\bbT)$. It is not difficult to see that \eqref{eqn:501} has a solution with $\alpha > 0$ if and only if $\lambda \leq \tfrac{d}{2\sqrt{d-1}}$. This gives the lower bound in part~\ref{item:CBOL1}.

To see that once~\eqref{eqn:501} holds, $M_t$ is a super-martingale, we argue as before. Given $\xi^0_t$ for some $t \geq 0$, the expected change in $M_t$ at the next transition of $\xi^0_\cdot$ is
\begin{equation}
	\frac{1}{(\lambda + 1)|\partial_\bbT \xi^0_t|}
		\sum_{v \in \xi_t, u \notin \xi_t} \lambda f(u) - f(v) \,.
\end{equation}
It is therefore enough to argue that for any finite $U \subseteq \bbT$,
\begin{equation}
\label{eqn:502}
	\sum_{v \in U, u \notin U} \lambda f(u) - f(v) \leq 0 \,.
\end{equation}
and by linearity, we can also assume that $U$ is connected.

The proof of~\eqref{eqn:502} will follow by induction on the size of $U$. It is easy to check that~\eqref{eqn:501} is necessary and sufficient for~\eqref{eqn:502} to hold for $U = \{x\}$.
Now suppose that it holds for all $U$ with $|U| \leq n$ and let $U$ be a set with $n$ vertices.
Choose some leaf $w \in U$ and set $U' := U \setminus \{w\}$. From the induction hypothesis we know
\begin{equation}
    \sum_{v \in U', u \notin U'} \lambda f(u)-f(v) \leq 0 \quad \text{and} \quad
    \sum_{u \sim w}\lambda f(u)-f(w) \leq 0 \,.
\end{equation}
Adding the two inequalities and letting $w'$ denote the parent of $w$, we obtain
\[
    (\lambda-1)(f(w)+f(w')) + \sum_{v \in U, u \notin U}\lambda f(u)-f(v) \leq 0
\]
But since $\lambda \geq 1$, this implies that~\eqref{eqn:502} holds as desired.

\section{Proof of Theorem~\ref{thm:Automorphism}}
\label{sec:ProofOfAuto}
Starting with part~\ref{part:AU1}, for $u \sim v \in \bbT$ and $t \geq 0$, denote by $e_{t}^{+}(u,v)$ the number of times $v$ becomes infected because of $u$, up to time $t$ and let $E_{t}^{+}(u,v) := \bbE\, e_{t}^{+}(u,v)$. If $\bbP(\xi_0 \in \cdot) \in \cI$, the latter does not depend on the choice of $u$, $v$ and we shall therefore just write $E_{t}^{+}$. Similarly, define $e_{t}^{-}(u,v)$ as the number of times $v$ becomes healthy because of $u$, up to time $t$ and let $E^-_t = E^-_t(u,v) := \bbE\, e_{t}^{-}(u,v)$. Since for all $v \in \bbT$,
\begin{equation}
	\sum_{u: u \sim v}e_{t}^{+}(u,v) - \sum_{u: u\sim v}e_{t}^{-}(u,v)=\pm 1 \,,
\end{equation}
by taking expectations we see that $-\frac{1}{d}\leq E_{t}^{+} - E_{t}^{-}\leq \frac{1}{d}$. On the other hand, $E_{t}^{+}(u,v) = \lambda E_{t}^{-}(v,u)$, as every time $\xi_\cdot^{\xi_0}(u)=+$ but $\xi_\cdot^{\xi_0}(v)=-$ and an infection or healing event occurs along the edge $\{u,v\}$, the probability of an infection of $v$ by $u$ is larger by a factor of $\lambda$ than the probability of a healing of $u$ by $v$ (formally, this is just a simple martingale argument). Thus $(1-\lambda^{-1})E_{t}^{+} \leq 1/d$ and so $E_{t}^{+}$ is bounded by a constant which does not depend on $t$. 

Consequently, $e_\infty^+(u,v) := \lim_{t \to \infty} e_t^+(u,v)$, which exists by monotonicity, has a bounded expectation and therefore must itself be bounded almost-surely. The same holds for $e_\infty^-(u,v)$ and we see that the number of sign flips at any vertex $v$ must be finite almost surely, as desired.

Turning to part~\ref{part:AU2}. Let $x$, $y$ be two neighboring vertices in $\bbT$ and $t \geq 0$.
Define 
\begin{equation}
	\rho_t := \bbP(\xi^{\xi_0}_t(x) = +) \quad , \qquad
	\delta_t := \bbP(\xi^{\xi_0}_t(x) = +, \,
			\xi^{\xi_0}_t(y) = -) \,.	
\end{equation}
Clearly, the above does not depend on the choice of $x$, $y$ since the initial distribution is automorphism-invariant. We claim that for all $t \geq 0$.
\begin{equation}
\label{eqn:126}
	\frac{d \rho_t}{dt} = d(\lambda-1) \delta_t \,.
\end{equation}
To see this, fix some $t, h \geq 0$ and for $u \sim v \in \bbT$, let
\begin{equation}
	f(u,v) := \big[ e^+_{t+h}(u,v) - e^+_{t}(u,v) \big] \, - \,
			  \big[ e^-_{t+h}(u,v) - e^-_{t}(u,v) \big]
			  	\quad , \quad F(u,v) := \bbE\, f(u,v) \,.
\end{equation}
Since $\bbT$ is transitive and unimodular, by the mass-transport principle (see e.g, \cite{lyons2005probability}),
\begin{equation}
	 \sum_{u : u \sim v} F(u,v) \, = \, \sum_{v : v \sim u} F(u,v) \,.
\end{equation}
The r.h.s. above is equal to $\bbE \tfrac12 \big( \xi_{t+h}(v) - \xi_t (v) \big) = 
	\rho_{t+h} - \rho_t$. The l.h.s. is
\begin{equation}
	\sum_{v: v \sim u} \bbE \big[ e^+_{t+h}(u,v) - e^+_{t}(u,v) \big] \, - \,
	\sum_{v: v \sim u} \bbE \big[ e^-_{t+h}(u,v) - e^-_{t}(u,v) \big] 
	= d \delta_t h (\lambda - 1) + O(h^2) \,.
\end{equation}
Equating the two sides, dividing by $h$ and taking $h \to 0$ we obtain~\eqref{eqn:126}.

Now from the previous part we know that $\xi_\infty^{\xi_0} := \lim_{t \to \infty} \xi_t^{\xi_0}$ exists almost-surely and its distribution is supported on $\{\emptyset, \bbT\}$. Since ergodicity is preserved in the (strong) limit, it follows that $\xi_\infty^{\xi_0}$ is either $\emptyset$ a.s. or $\bbT$ a.s. If $\bbP(\xi_0 \in \cdot) \neq \delta_\emptyset$, then $\rho_0 > 0$ and since the r.h.s. of~\eqref{eqn:126} is non-negative it must be that also $\rho_\infty := \bbP(\xi_\infty^{\xi_0}(x) = +) > 0$. This leaves only the option $\xi_\infty^{\xi_0} = \bbT$ almost surely, which implies $\bbP(\Omega_c^{\xi_0}) = 1$.

It remains to show that infinite infected components must form in finite time. We first show that if at time $t \geq 0$ all components are finite almost-surely, then 
\begin{equation}
\label{eqn:130}
	\delta_t \geq C \rho_t \,.
\end{equation}
for some $C > 0$. This follows again from the mass-transport principle. Denote by $\partial_\bbT^- U$ the internal vertex boundary of $U$, that is, the set of vertices in $U$ with neighbors that are not in $U$. Let every infected vertex $v$ under $\xi^{\xi_0}_t$, send one unit of mass to the interior vertex boundary of its component, divided equally among these boundary vertices. The expected amount of mass a vertex sends is $\rho_t$. Thus, by the mass-transport principle, the expected amount of mass a vertex receives is $\rho_t$ as well. 

On the other hand, a vertex receives a positive mass if and only if it is infected and
lies in $\partial^-_\bbT U$ where $U$ is the infected component to which it belongs. In this case it receives a total mass of $|U|/|\partial^-_\bbT U|$. Since $\bbT$ is non-amenable, this ratio is bounded above by a constant $C' < \infty$. If we let $\beta_t$ be the probability a given vertex is in the interior boundary of an infected component of $\xi_t$, then these considerations lead to the inequality $\beta_t \geq C'^{-1} \rho_t$. Since $\delta_t \geq \beta_t / d$, we have shown~\eqref{eqn:130}.

Now suppose that $\bbP(\xi_0 \in \cdot)$ is as in part~\ref{part:AU2} of the theorem and that for all time $t \geq 0$, the probability of the existence of an infinite infected component in $\xi^{\xi_0}_t$ is $0$. Then~\eqref{eqn:126} and~\eqref{eqn:130} together imply that 
$\rho_t \geq \rho_0 e^{C'' t}$
for some $C'' > 0$. Since $\rho_0 > 0$, this would imply $\rho_{t} > 1$ for some $t > 0$, which is a contradiction. Therefore, there is some $t \geq 0$, for which $\xi^{\xi_0}_t$ has infinite components with positive probability. Since the latter event is automorphism-invariant and since ergodicity is carried over to $\xi_t^{\xi_0}$ for every $t > 0$, the latter probability must be $1$.

\section*{Acknowledgments}
The authors would like to thank Tom Liggett, Noam Berger and Elchanan Mossel for many useful discussions.

\bibliographystyle{abbrv}
\bibliography{VoterSurvival}

\end{document}